\documentclass[11pt]{article}

\usepackage{mathtools}
\usepackage{amsmath}
\usepackage{amsthm}
\usepackage{amsfonts}
\usepackage{amssymb}
\usepackage{latexsym}
\usepackage{mathrsfs}

\usepackage{tabularx}
\usepackage{diagbox}

\newcommand{\vt} { {\bf t}}
\newcommand{\vk} { {\bf k}}
\newcommand{\vv} { {\bf v}}
\newcommand{\vx} { {\bf x}}

\DeclareMathOperator{\DR}{DR}
\DeclareMathOperator{\ar}{arc}

\usepackage[noautoscale]{youngtab}

\usepackage{fullpage}

\usepackage{color, colortbl}
\usepackage{array}
\usepackage{multirow}

\usepackage{tikz}
\usepackage{subcaption}

\usepackage{dutchcal}

\usepackage{enumitem}
\setlist[enumerate, 1]{label=(\roman*)}

\numberwithin{equation}{section}

\newtheorem{thm}{Theorem}[section]
\newtheorem{lem}[thm]{Lemma}
\newtheorem{prop}[thm]{Proposition}
\newtheorem{cor}[thm]{Corollary}

\newtheorem{rem}[thm]{Remark}

\newtheorem{exa}[thm]{Example}

\newcommand{\cred}{\color{red}}
\newcommand{\cblue}{\color{blue}}

\def\bN{\mathbb N}

\def\bP{\mathbb P}

\def\cR{\mathcal{R}}
\def\cW{\mathcal{W}}

\def\cP{\mathcal{P}}

\def\oQ{\overline{Q}}

\newcommand{\cp}{\mathcal{p}}
\newcommand{\ccr}{\mathcal{r}}
\newcommand{\cc}{\mathcal{c}}

\usepackage[colorlinks=true,
linkcolor=webgreen,
filecolor=webbrown,
citecolor=webgreen]{hyperref}

\definecolor{webgreen}{rgb}{0,.5,0}
\definecolor{webbrown}{rgb}{.6,0,0}

\usepackage[symbol]{footmisc}

\begin{document}

\title{Decreasing Runs in Quasi-Stirling Permutations of Multisets\thanks{Hanqian Fang was partially supported by the Strategic Priority Research Program of the Chinese Academy of Sciences (No. XDB0510201) and the NSFC grant (No. 12271511).
}}

\author{Hanqian Fang\footnote{KLMM, Academy of Mathematics and Systems Science, Chinese Academy of Sciences, Beijing 100190, P. R. China.  Email: fanghanqian22@mails.ucas.ac.cn.}}

\date{\today}

\maketitle

\noindent\textbf{Abstract.} As a natural extension of Stirling permutations, quasi-Stirling permutations are multipermutations $\pi$ with the property that for any subsequence $\pi_{j_1}\pi_{j_2}\pi_{j_3}\pi_{j_4}$ satisfying $\pi_{j_1}=\pi_{j_3}$ and $\pi_{j_2}=\pi_{j_4}$, we have $\pi_{j_1}=\pi_{j_2}$. Using a bijective construction, Yan, Yang, Huang and Zhu showed that the joint distribution of ascents, descents and plateaux over quasi-Stirling permutations of a multiset $M=\{1^{k_1},2^{k_2},\ldots,n^{k_n}\}$ coincides with that over the multiset $M'=\{1^{k_1+\cdots+k_n-n+1},2,\ldots,n\}$. In this paper, we prove that the same invariance of distribution holds for decreasing runs, and consequently for all decreasing consecutive patterns. To this end, following the Yan-Yang-Huang-Zhu approach, we construct a multiplicity-redistribution bijection that preserves decreasing runs, thereby reducing the computation of joint distribution of decreasing consecutive patterns over quasi-Stirling permutations from $M$ to $M'$. Together with the classical run theorem, our bijection leads to explicit recurrence relations and generating functions for the distribution functions of these statistics over quasi-Stirling permutations.

\qquad\\

\noindent {\bf AMS Classification 2010:} 05A05; 05A15; 05A19


\tableofcontents


\newpage

\section{Introduction}

A \textit{Stirling permutation} of a multiset $M=\{1^{k_1},2^{k_2},\ldots,n^{k_n}\}$ is a multipermutation $\pi$ such that for every triple of indices $j_1<j_2<j_3$, the equality $\pi_{j_1}=\pi_{j_3}$ implies $\pi_{j_2}\ge \pi_{j_1}$. Introduced by Gessel and Stanley~\cite{Gessel1978} in 1978 for the case $k_1=k_2=\cdots=k_n=2$, these \textit{classical Stirling permutations} were later shown by Janson~\cite{Janson2008} to be in bijection with increasing edge-labeled plane trees. Removing the increasing labeling condition extends this bijection to edge-labeled plane trees and a broader class of permutations, called quasi-Stirling permutations by Archer, Gregory, Pennington, and Slayden~\cite{Archer2019-1}. In general, a multipermutation $\pi$ of $M$ is called a \textit{quasi-Stirling permutation} if there do not exist indices $j_1<j_2<j_3<j_4$ such that $\pi_{j_1}=\pi_{j_3}$ and $\pi_{j_2}=\pi_{j_4}$ while $\pi_{j_1}\ne\pi_{j_2}$. When $k_1=k_2=\cdots=k_n=2$, these are precisely the \emph{classical quasi-Stirling permutations} studied in~\cite{Archer2019-1}. An alternative perspective is to view $\pi$ as a labeled matching on $[K]$, where $K=k_1+k_2+\cdots+k_n$. For each $\ell\in[n]$, we connect every pair of indices $j_1<j_2$ satisfying $\pi_{j_1}=\pi_{j_2}=\ell$ by an arc, which we denote by $\ar(\ell;j_1,j_2)$ (see~\cite[Exercise~6.19(o)]{StanleyBook2} for the classical case). Under this interpretation, $\pi$ is a quasi-Stirling permutation if and only if any two arcs with distinct labels are noncrossing. Let $\oQ_M$ denote the set of quasi-Stirling permutations of $M$. For example, if $M=\{1^2,2^3\}$, then $\pi=21221$ is not a quasi-Stirling permutation, as illustrated in Figure~\ref{FIG:labeled matching}. Indeed, $\oQ_M=\{12221,22112,21122,11222,22211\}$.

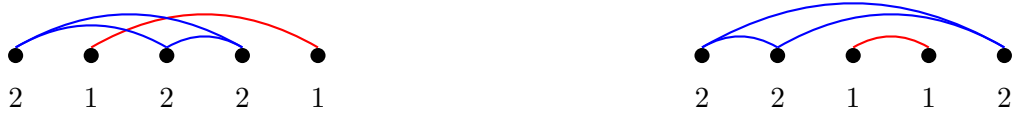
\begin{figure}[htbp]
	\centering
	\begin{subfigure}[b]{0.45\textwidth}
		\centering
		\begin{tikzpicture}[
		every node/.style={font=\sffamily},
		arc/.style={thick}
		]
		\foreach \x in {1,...,5} {
			\node[circle, fill=black, inner sep=1pt, minimum size=2mm] (p\x) at (\x, 0) {};
		}
		\node[below] at (1,-0.3) {$2$};
		\node[below] at (2,-0.3) {$1$};
		\node[below] at (3,-0.3) {$2$};
		\node[below] at (4,-0.3) {$2$};
		\node[below] at (5,-0.3) {$1$};
		\draw[arc, red, bend left=30] (p2.north) to (p5.north);
		\draw[arc, blue, bend left=30] (p1.north) to (p3.north);
		\draw[arc, blue, bend left=30] (p3.north) to (p4.north);
		\draw[arc, blue, bend left=30] (p1.north) to (p4.north);
		\end{tikzpicture}
	\end{subfigure}
	\hfill
	\begin{subfigure}[b]{0.45\textwidth}
		\centering
		\begin{tikzpicture}[
		every node/.style={font=\sffamily},
		arc/.style={thick}
		]
		\foreach \x in {1,...,5} {
			\node[circle, fill=black, inner sep=1pt, minimum size=2mm] (p\x) at (\x, 0) {};
		}
		\node[below] at (1,-0.3) {$2$};
		\node[below] at (2,-0.3) {$2$};
		\node[below] at (3,-0.3) {$1$};
		\node[below] at (4,-0.3) {$1$};
		\node[below] at (5,-0.3) {$2$};
		\draw[arc, red, bend left=30] (p3.north) to (p4.north);
		\draw[arc, blue, bend left=30] (p1.north) to (p2.north);
		\draw[arc, blue, bend left=30] (p1.north) to (p5.north);
		\draw[arc, blue, bend left=30] (p2.north) to (p5.north);
		\end{tikzpicture}
	\end{subfigure}
	\caption{Labeled matching representations. Red arcs: label $1$; blue arcs: label $2$.}
	\label{FIG:labeled matching}
\end{figure}

Patterns play a prominent role in the study of permutations. Here we focus on \textit{consecutive patterns} and refer to~\cite{Kitaev2011} for a comprehensive treatment. A consecutive pattern $p=\underline{p_1p_2\cdots p_s}$ is a word over an alphabet $[v]$ using every letter. We say that a permutation 
$\pi = \pi_1\pi_2\cdots \pi_K \in S_M$ \emph{contains} an occurrence of the pattern $p$ 
if there exists a factor $\pi_j \pi_{j+1} \cdots \pi_{j+s-1}$ of $\pi$ such 
that for any two indices $j_1, j_2 \in [s]$, we have 
$\pi_{j+j_1-1} < \pi_{j+j_2-1}$ (resp., $\pi_{j+j_1-1} = \pi_{j+j_2-1}$) if and only if $p_{j_1} < p_{j_2}$ (resp., $p_{j_1} = p_{j_2}$). 
If $\pi$ contains no occurrence of $p$, we say that $\pi$ \emph{avoids} $p$. For patterns $p^{(1)},\ldots,p^{(\lambda)}$, let $(p^{(1)},\ldots,p^{(\lambda)})\pi$ denote the tuple of their occurrence counts in $\pi$. Occurrences of $\underline{11}$, $\underline{12}$, and $\underline{21}$ are traditionally called \emph{plateaux} (or \emph{levels}), \emph{ascents}, and \emph{descents}, respectively. For $s\ge1$, we call the pattern $\underline{s\cdots21}$ a \emph{decreasing consecutive pattern}. The \emph{decreasing runs} of $\pi$ are its maximal consecutive factors with no ascents or plateaux. Let $\DR(\pi)=(m_1,\ldots,m_n)$, where $m_j$ is the number of decreasing runs of length $j$ in $\pi$. It is evident that the tuples $(\underline{1},\underline{21},\ldots,\underline{n\cdots21})\pi$ and $\DR(\pi)$ are in one-to-one correspondence via the equations
\begin{align} \label{EQ:p-DR}
\underline{s\cdots21}(\pi)=\sum_{j=s}^{n}(j-s+1)m_j \qquad \text{for }1\le s\le n.
\end{align}
For instance, for $M=\{1^3,2,3^2,4,5^2\}$ and $\pi=115315432$, we have $(\underline{1},\underline{21},\ldots,\underline{54321})\pi=(9,5,3,1,0)$ and $\DR(\pi)=(2,0,1,1,0)$, consistent with~\eqref{EQ:p-DR}.

Consecutive patterns in Stirling permutations have received considerable attention. Bóna~\cite{Bona2008} proved that ascents, plateaux and descents are equidistributed over classical Stirling permutations. Descent polynomials for general Stirling permutations were studied by Brenti~\cite{Brenti1989-2,Brenti1998-3} and Dzhumadil'daev and Yeliussizov~\cite{Dzhumadil2014-5}. Lin, Ma and Zhang~\cite{Lin2021} established partial $\gamma$-positivity for their enumerative polynomials with respect to the statistics of plateaux, descents and ascents. Some of these results have been extended to quasi-Stirling permutations. Elizalde~\cite{Elizalde2021} studied descent polynomials for classical quasi-Stirling permutations and extended the Archer–Gregory–Pennington–Slayden bijection to the equimultiplicity case. Yan and Zhu~\cite{YanZhu2022} further generalized this to arbitrary multisets, obtaining full descent polynomials. Later, Yan, Yang, Huang and Zhu~\cite{YanYangHuangZhu2022} constructed a bijection between quasi-Stirling permutations of $M$ and those of $M'=\{1^{K-n+1},2,\ldots,n\}$, where $K=k_1+k_2+\cdots+k_n$. This elegant bijection shows that the joint distribution of ascents, descents and plateaux is invariant under redistributions of letter multiplicities, thereby reducing the related enumeration on $M$ to the case $M'$.

This reduction naturally raises the question of whether other combinatorial statistics for quasi-Stirling permutations admit a similar property. In this paper, we give an affirmative answer. Following the spirit of the Yan–Yang–Huang–Zhu bijection, we construct a new bijection that preserves the distribution of decreasing runs. Consequently, the joint distribution of all decreasing consecutive patterns in quasi-Stirling permutations coincides with that for multipermutations with at most one repeated letter, and hence is independent of multiplicity redistribution.

\begin{thm}\label{THM:bijections-run-pattern}
	Let $M=\{1^{k_1},2^{k_2},\ldots,n^{k_n}\}$ be a multiset with $n\geq 1$ and $k_1,k_2,\ldots,k_n\in\mathbb P$. Let 
	$
	M'=\{1^{K-n+1},2,\ldots,n\}
	$ with $K=k_1+k_2+\cdots +k_n$.
	Then there exists a bijection $\Psi$ between $\overline{Q}_M$ and $\overline{Q}_{M'}$ such that for any $\pi\in \overline{Q}_M$,
	$$
	\DR(\pi)=\DR(\Psi(\pi))\quad\text{and}\quad(\underline{1},\underline{21},\ldots,\underline{n\cdots21})\pi=(\underline{1},\underline{21},\ldots,\underline{n\cdots21})\Psi(\pi).
	$$
\end{thm}

Building on this, we further study the distribution function of decreasing runs over quasi-Stirling permutations,
\[
\ccr(\vt;\oQ_M):=\sum_{\pi\in\oQ_M}\mathbf{t}^{\DR(\pi)},
\]
where $\vt=(t_1,t_2,\ldots)$ and $\mathbf{t}^{\DR(\pi)}=t_1^{m_1}t_2^{m_2}\cdots t_n^{m_n}$ for $\DR(\pi)=(m_1,m_2,\ldots,m_n)$. We express these functions
in terms of $\ccr(\vt;n,k):=\ccr(\vt;\oQ_{M'})$, where $M'=\{1^k,2,\ldots,n\}$,
and present the recurrence relation for $\ccr(\vt;n,k)$ together with its exponential generating function defined by
\begin{align}\label{EQ:def-cR}
\cR(\vt;x,y):=\sum_{n\ge1,\,k\ge0}\frac{\ccr(\vt;n,k)}{(n-1)!}x^n y^k.
\end{align}
Our derivation relies on the classical run theorem for multipermutations in~\cite{Gessel1977,Jackson1977,Goulden1983} and its specialization in~\cite{Yang2020}, where weight functions $W_n(\vt)$ are defined by the expansion
\begin{align}\label{EQ:W_n}
1-\sum_{n=1}^{\infty} W_n(\vt) z^n
= \left(1+\sum_{n=1}^{\infty} t_n z^n\right)^{-1}.
\end{align}
Moreover, let $\cW(\vt;x):=\sum_{n=1}^{\infty}W_n(\vt)x^n/n!$ denote the exponential generating function of $W_n(\vt)$.

\begin{thm}\label{THM:dis-run}
	Let $M=\{1^{k_1},2^{k_2},\ldots,n^{k_n}\}$ be a multiset with $n\geq 1$ and $k_1,k_2,\ldots,k_n\in\mathbb P$. Let 
	$K=k_1+k_2+\cdots +k_n$. Then
	\[
	\ccr(\vt;\oQ_M)=\ccr(\vt;n,K-n+1).
	\]
	Furthermore, $\ccr(\vt;n,k)$ satisfies the recurrence relations
	\begin{align*}
	\ccr(\vt;n,k)=
	\sum_{m=1}^{n}W_m(\vt)\left(\binom{n-1}{m}\ccr(\vt;n-m,k)+\binom{n-1}{m-1}\ccr(\vt;n-m+1,k-1)\right)
	\end{align*}
	for $n\geq1$ and $k\geq1$, together with $\ccr(\vt;n,0)=\ccr(\vt;n-1,1)$ for $n\geq2$,
	and initial conditions $\ccr(\vt;1,0)=1$ and $\ccr(\vt;n,k)=0$ whenever $n\leq0$ or $k<0$.
	The exponential generating function of $\ccr(\vt;n,k)$ is given by
	\begin{align*}
	\cR(\vt;x,y)=\frac{x}{1-\bar{\cW}(\vt;x)-y\cdot \bar{\cW}'(\vt;x)},
	\end{align*}
	where the prime denotes differentiation with respect to $x$.
\end{thm}

Analogous results for the distribution functions of decreasing consecutive patterns over quasi-Stirling permutations, including the cases of joint distributions, single patterns and pattern avoidance, are obtained in Theorems~\ref{THM:enum-joint-pattern},~\ref{THM:enum-p-dis} and~\ref{THM:avoid-p}, respectively. A related problem is to determine the maximum possible number of a given decreasing consecutive pattern in a quasi-Stirling permutation, and to enumerate the permutations that achieve this maximum. This question was raised by Archer et al.~\cite{Archer2019-1} and answered by Elizalde~\cite{Elizalde2021} for descents in classical quasi-Stirling permutations, and later generalized by Yan et al.~\cite{YanYangHuangZhu2022} to descents in general quasi-Stirling permutations. In this paper, we present the corresponding result for any non-trivial decreasing consecutive pattern.

\begin{cor}\label{COR:most-p}
	Let $s>1$ be an integer. Let $M=\{1^{k_1},2^{k_2},\ldots,n^{k_n}\}$ be a multiset with $n\geq s$ and $k_1,k_2,\ldots,k_n\in\mathbb P$. Let 
	$K=k_1+k_2+\cdots +k_n$. Then every quasi-Stirling permutation of $M$ contains at most $n-s+1$ occurrences of the pattern $\underline{s\cdots21}$. Moreover, the number of permutations $\pi\in\oQ_M$ with $(\underline{s\cdots21})\pi=n-s+1$ is $(K-n+1)^{n-1}$ if $s=2$, and $K-n+1$ if $s>2$.
\end{cor}

The paper is organized as follows. Section~\ref{SEC:pre} introduces the necessary terminology and recognition criteria for quasi-Stirling permutations of multisets. In Section~\ref{SEC:bijections}, we construct a bijection between quasi-Stirling permutations of two multisets under multiplicity redistribution, and then complete the proof of Theorem~\ref{THM:bijections-run-pattern}. Using this bijection, Section~\ref{SEC:dis} computes the corresponding distribution functions for decreasing runs and decreasing consecutive patterns in two subsections. Section~\ref{SEC:dis-DR} is devoted to the proof of Theorem~\ref{THM:dis-run}. Section~\ref{SEC:dis-p} contains the proof of Corollary~\ref{COR:most-p} and also presents Theorems~\ref{THM:enum-joint-pattern},~\ref{THM:enum-p-dis}, and~\ref{THM:avoid-p}.

\section{Definitions and Recognition Criteria}\label{SEC:pre}

In this section, we introduce additional definitions and notation for quasi-Stirling permutations of multisets, along with some useful criteria for their recognition, in preparation for the bijective construction in Section~\ref{SEC:bijections}.
 
We first recall some basic definitions for permutations of a multiset, for the convenience of the reader. For any positive integer $m$, denote the set $\{1,2,\ldots,m\}$ by $[m]$. 
For a non-negative integer vector $(k_1,k_2,\ldots,k_n)\in\bN^n$, let $\{1^{k_1},2^{k_2},\ldots,n^{k_n}\}$ denote the \textit{multiset} $$\{\underbrace{1,\ldots,1}_{k_1},\underbrace{2,\ldots,2}_{k_2},\ldots,\underbrace{n,\ldots,n}_{k_n}\}$$ on the set of positive integers $\bP$. For a multiset $M=\{1^{k_1},2^{k_2},\ldots,n^{k_n}\}$, the integer $k_\ell$ is called the \textit{multiplicity} of the letter $\ell$ and the sum $k_1+k_2+\cdots+k_n$ is called the \textit{size} of $M$, denoted by $|M|$. A \textit{permutation} (also known in the literature as a \textit{multipermutation}) of $M$ is a sequence of \textit{length} $|M|$ where each \textit{letter} $\ell$ occurs exactly $k_\ell$ times for all $1\leq \ell\leq n$. We use $S_M$ to denote the set of all permutations of $M$. For a permutation $\pi = \pi_1 \pi_2 \cdots \pi_K$, a subsequence $\pi_j \pi_{j+1} \cdots \pi_{j'}$ with $1 \le j \le j' \le K$ is called a \emph{factor} of $\pi$.


We now introduce some terminology for bijective construction. Let $M=\{1^{k_1},2^{k_2},\ldots,n^{k_n}\}$ be a multiset and fix a letter $i\in[n]$. The \textit{remaining sequence} $R(\pi,i)$ of $i$ in $\pi$ is the subsequence obtained from $\pi$ by deleting all occurrences of $i$ and all letters between consecutive occurrences of $i$. Assuming $k_i>1$, for $\pi\in S_M$ and $k\in[k_i-1]$, the \textit{descendant sequence} $D(\pi,i,k)$ of the $k$-th $i$ in $\pi$ is the factor of $\pi$ strictly between the $k$-th and the $(k+1)$-th occurrences of $i$. The \textit{child sequence} $C(\pi,i,k)$ of the $k$-th $i$ in $\pi$ is defined via an iterative procedure. Start with $\alpha:=D(\pi,i,k)$ and $\beta:=()$. While $\alpha$ is nonempty, repeat the following operations.
\begin{enumerate}
	\item append the first letter $j$ of $\alpha$ to $\beta$;
	\item replace $\alpha$ by its factor after the last occurrence of $j$ in $\alpha$.
\end{enumerate}
The process terminates because the length of $\alpha$ strictly decreases in each step. When $\alpha=()$, set $C(\pi,i,k):=\beta$. The \textit{block set} $B(\pi,i)$ is defined as the collection of all nonempty sequences among $R(\pi,i)$ together with all descendant sequences whenever such sequences exist, i.e.,
\[
B(\pi,i)=\{R(\pi,i)\}\cup\{D(\pi,i,k):1\le k\le k_i-1\}\setminus\{()\}.
\]
While these definitions are valid for any multipermutation, our main focus is on their application to quasi-Stirling permutations. 

\begin{exa}\label{EX:childseq}
	Let $M=\{1^3,2,3^2,4^2,5^5,6,7,8\}$ and $\pi=5285146411733555\in\oQ_{M}$. Then we get $R(\pi,1)=5285733555$ and $R(\pi,5)=()$. Now fix $i=5$. We demonstrate the computation of the child sequence of the second occurrence of $5$ in $\pi$ as follows. First, we have $D(\pi,i,2)=146411733$.
	\begin{table*}[hbtp]
		\newcolumntype{Y}{>{\centering\arraybackslash}X}
		\begin{tabularx}{\textwidth}{c|Y|Y|Y|Y}
			\hline
			\diagbox{Sequences}{Steps} & 0        & 1   & 2  & 3    \\ \hline
			$\alpha$ & 146411733 & 733 & 33 & $()$ \\ \hline
			$\beta$  & $()$     & 1   & 17 & 173  \\ \hline			
		\end{tabularx}
		\caption{The process to obtain $C(\pi,i,2)$.}
		\label{TAB1}
	\end{table*}
	
	\noindent As shown in Table~\ref{TAB1}, after three iterations $\alpha$ becomes empty and the process stops. Hence, $C(\pi,i,2)=\beta=173$. Similarly, we have $D(\pi,i,1)=C(\pi,i,1)=28$ and $D(\pi,i,3)=C(\pi,i,3)=D(\pi,i,4)=C(\pi,i,4)=()$. Therefore, $B(\pi,i)=\{28,146411733\}$.
\end{exa}

Some of the terminology introduced above is motivated by the combinatorial 
interpretation of quasi-Stirling permutations in terms of ordered labeled trees. 
In the bijection of Yan and Zhu~\cite{YanZhu2022} from quasi-Stirling permutations 
to ordered labeled trees with certain restrictions, the factor $D(\pi,i,k)$ 
corresponds to the subtree rooted at the $k$-th child of the odd vertex labeled $i$ 
(more precisely, it corresponds to the labels of all vertices in that subtree except 
the root), and the sequence $C(\pi,i,k)$ corresponds to the labels of the children 
of that same $k$-th child. This explains our choice of the terms ``descendant sequence" 
and ``child sequence".
From the perspective of labeled matchings, these terms also have natural interpretations. 
By the equivalent definition of quasi-Stirling permutations, it suffices to consider 
only the arcs between consecutive occurrences of each letter. 
In this setting, the factor $D(\pi,i,k)$ corresponds to the segment enclosed by the 
$k$-th arc labeled $i$, and the sequence $C(\pi,i,k)$ corresponds to the labels of 
the outermost arcs and isolated vertices contained within that same arc, where the same label is recorded only once. For an illustration, see Figure~\ref{FIG:labeled matching_2}, 
which displays the labeled matching representation for the permutation in 
Example~\ref{EX:childseq}.
\begin{figure}[htbp]
	\centering
	\begin{tikzpicture}[
	every node/.style={font=\sffamily},
	arc/.style={thick}
	]
	
	\foreach \x in {1,...,16} {
		\node[circle, fill=black, inner sep=1pt, minimum size=2mm] (p\x) at (\x, 0) {};
	}
	
	\node[below] at (1, -0.3) {$5$};
	\node[below] at (2, -0.3) {$2$};
	\node[below] at (3, -0.3) {$8$};
	\node[below] at (4, -0.3) {$5$};
	\node[below] at (5, -0.3) {$1$};
	\node[below] at (6, -0.3) {$4$};
	\node[below] at (7, -0.3) {$6$};
	\node[below] at (8, -0.3) {$4$};
	\node[below] at (9, -0.3) {$1$};
	\node[below] at (10, -0.3) {$1$};
	\node[below] at (11, -0.3) {$7$};
	\node[below] at (12, -0.3) {$3$};
	\node[below] at (13, -0.3) {$3$};
	\node[below] at (14, -0.3) {$5$};
	\node[below] at (15, -0.3) {$5$};
	\node[below] at (16, -0.3) {$5$};

	\draw[arc, black, bend left=30] (p1.north) to (p4.north);
	\draw[arc, black, bend left=30] (p4.north) to (p14.north);
	\draw[arc, black, bend left=30] (p14.north) to (p15.north);
	\draw[arc, black, bend left=30] (p15.north) to (p16.north);
	
	\draw[arc, red, bend left=30] (p5.north) to (p9.north);
	\draw[arc, red, bend left=30] (p9.north) to (p10.north);
	
	\draw[arc, blue, bend left=30] (p6.north) to (p8.north);
	
	\draw[arc, green, bend left=30] (p12.north) to (p13.north);

	\end{tikzpicture}
	\caption{Labeled matching representation of $5285146411733555$, with distinct colors representing distinct arc labels.}
	\label{FIG:labeled matching_2}
\end{figure}
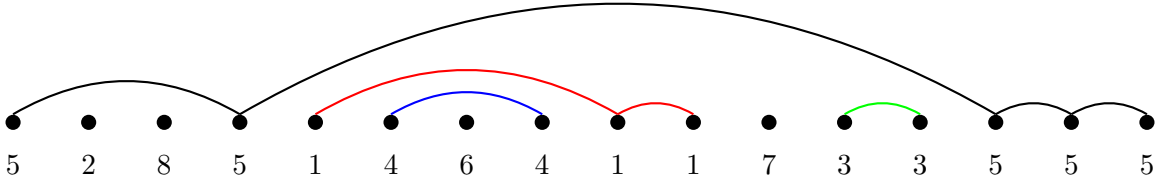

\noindent For instance, from the figure one can read off that the sequence $C(\pi,5,2)$ consists of $1$ (the red arc), $7$ (the isolated vertex), and $3$ (the green arc).
Furthermore, as we will show below, descendant and remaining sequences also give a convenient characterization of quasi-Stirling permutations.

\begin{prop}\label{PROP:test-1}
	Let $M=\{1^{k_1},2^{k_2},\ldots,n^{k_n}\}$ be a multiset and let $i$ be a letter in $M$. A permutation $\pi\in S_M$ is a quasi-Stirling permutation if and only if every sequence in $B(\pi,i)$ is a quasi-Stirling permutation, and no two of them share a common letter.
\end{prop}
\begin{proof}
	The case of $k_i=1$ is trivial, so we assume that $k_i>1$.
	
	For necessity, consider two non-empty child sequences $D(\pi,i,k)$ and $D(\pi,i,k')$ with $1\le k<k'\le k_i-1$. If some letter occurs in both sequences, then the arc between its two occurrences crosses the arc between the $k$-th and $k'$-th occurrences of $i$, a contradiction. A similar argument applies to a child sequence $D(\pi,i,k)$ and the remaining sequence $R(\pi,i)$.
	
	For sufficiency, we conversely suppose that there exist two crossing arcs $\ar(\ell_1;j_1,j_2)$ and $\ar(\ell_2;j_3,j_4)$ with $\ell_1\neq\ell_2$ in $\pi$. Then at least one of $\ell_1,\ell_2$ is not $i$. So we can assume without loss of generality that $\pi_{j_1}$ lies in some sequence $\alpha\in B(\pi,i)$. Since $\ell_1$ cannot appear in any other sequence, $\pi_{j_2}$ must lie in $\alpha$. By the crossing condition, either $\pi_{j_3}$ or $\pi_{j_4}$ must also lie in $\alpha$, and consequently all four elements belong to $\alpha$. But this contradicts the fact that $\alpha$ is itself a quasi-Stirling permutation. Therefore, $\pi$ has no pair of crossing arcs with distinct labels, and hence is a quasi-Stirling permutation.
\end{proof}

The above proposition can be applied recursively. For two distinct letters $i$ and $i'$, we define the \textit{block set} $B(\pi,(i,i'))$ of $(i,i')$ in $\pi$ to be
$
B(\pi,(i,i')):=\bigcup_{\alpha\in B(\pi,i)} B(\alpha,i').
$
In fact, this set essentially consists of the blocks obtained by slicing $\pi$ according to all occurrences of the letters $i$ and $i'$, except that some of them may need to be merged.
For instance, continuing from Example~\ref{EX:childseq}, we have $B(\pi,(5,4))=\{28,111733,6\}$, where the blocks $1$ and $11733$ are merged. Let $S(\pi,(i,i'))$ denote the subsequence of $\pi$ consisting of all occurrences of $i$ and $i'$. We then define
\[
\bar{B}(\pi,(i,i')):=B(\pi,(i,i'))\cup \{S(\pi,(i,i'))\}.
\]
Applying Proposition~\ref{PROP:test-1} twice yields the following refined characterization.

\begin{cor}\label{COR:test-2}
	Let $M=\{1^{k_1},2^{k_2},\ldots,n^{k_n}\}$ be a multiset and let $i, i'$ be two distinct letters in $M$. A permutation $\pi\in S_M$ is a quasi-Stirling permutation if and only if every sequence in $\bar{B}(\pi,(i,i'))$ is a quasi-Stirling permutation, and no two of them share a common letter.
\end{cor}

This corollary will serve as a key tool in Section~\ref{SEC:bijections} for verifying that the resulting permutations produced by the mapping are indeed quasi-Stirling permutations.

\section{Bijections Preserving the Distribution of Decreasing Runs}\label{SEC:bijections}

The objective of this section is to construct a bijection between the set of quasi-Stirling permutations of a multiset and that of a modified multiset, obtained from the original by redistributing the multiplicities of its letters. Furthermore, we will show that this bijection preserves the distribution of decreasing runs and thus complete the proof of Theorem~\ref{THM:bijections-run-pattern}.

Our construction is inspired by the bijection of Archer et al.~\cite{Archer2019-1} for the classical case, as well as its extensions to more general multisets by Elizalde~\cite{Elizalde2021} and by Yan and Zhu~\cite{YanZhu2022}. These bijections establish a correspondence between quasi-Stirling permutations and ordered labeled trees with certain restrictions. Building on these results, Yan et al.~\cite{YanYangHuangZhu2022} further constructed a bijection between quasi-Stirling permutations of distinct multisets that is shown to preserve ascents, descents and plateaux. Our mapping follows a similar spirit, but is adapted to preserve the distribution of decreasing runs.

Let $M=\{1^{k_1},2^{k_2},\ldots,n^{k_n}\}$ be a multiset with $n\geq2$ and $k_1,k_2,\ldots,k_n\in\bP$ and set $K:=|M|$. Assume there exists a letter $i \in [n]\setminus\{1\}$ with $k_i > 1$, and take $i$ to be the maximal such integer. Throughout the remainder of this section, unless stated otherwise, we maintain this assumption and fix the multiset $M$ and the letter $i$.

We first introduce two auxiliary mappings. With the convention $\pi_0=0$, for any index $j\in[K]$, define $\phi(j)\in[j-1]\cup\{0\}$ to be the largest index such that $\pi_{\phi(j)}\leq i$. Note that $\phi$ is monotone on the set $\{j\in[K]: \pi_j\le i\}$. More precisely, if $j_1< j_2$ and $\pi_{j_1},\pi_{j_2}\leq i$, then $\phi(j_1)<j_1\leq\phi(j_2)<j_2$. For instance, using the setting in Example~\ref{EX:childseq}, take $j_1 = 1$, $j_2 = 7$ and $j_3 = 8$. Then $\phi(j_1)=0<j_1<\phi(j_2)=\phi(j_3)=6<j_2<j_3$. The second mapping $\theta$ acts on a multipermutation by deleting all letters greater than $i$. Continuing with Example~\ref{EX:childseq}, we have $\theta(\pi)=5251441133555$.

Now we are ready to present the desired mapping. Let $M'=\{1^{k_1},\ldots,(i-1)^{k_{i-1}+1},i^{k_{i}-1},\ldots,n^{k_n}\}$. For any $\pi=\pi_1\pi_2\cdots\pi_K\in\oQ_M$, we define $\Phi(\pi)\in S_{M'}$ according to the following two cases. 

\noindent{\bf Case 1.} Suppose that $\theta(\pi)$ contains the factor $i(i-1)$.

\noindent{\bf Case 1.1.} If there exists some $i$ occurring after the last $i-1$ in $\pi$, then $\Phi(\pi)$ is obtained from $\pi$ by replacing the first such $i$ with $i-1$.

\noindent{\bf Case 1.2.} Otherwise, $\Phi(\pi)$ is obtained by replacing the first $i$ in $\pi$ with $i-1$.

\noindent{\bf Case 2.} Suppose that $\theta(\pi)$ does not contain the factor $i(i-1)$. 

\noindent\textbf{Case 2.1.} Assume that $D(\pi,i,k_i-1)$ does not contain $i-1$. Let $\pi_{j_1}$ and $\pi_{j_2}$ be the last two occurrences of $i$ with $j_1<j_2$, and let $\pi_{j_3}$ be the last occurrence of $i-1$. Then either $j_1<j_2<j_3$ or $j_3<j_1<j_2$, which means that the case $\phi(j_1)<\phi(j_3)\le \phi(j_2)$ cannot occur. Set $\alpha=\pi_{\phi(j_1)+1}\pi_{\phi(j_1)+2}\cdots\pi_{\phi(j_2)}$. Clearly, $\alpha$ contains exactly one occurrence of $i$. Then $\Phi(\pi)$ is obtained from $\pi$ by moving the factor $\alpha$ to the position immediately after $\pi_{\phi(j_3)}$, with that unique occurrence of $i$ in $\alpha$ replaced by $i-1$. We present the explicit forms of $\pi$ and $\Phi(\pi)$ for the case $j_1 < j_2 < j_3$ as follows, where the ordinal number in the first row indicates the occurrence count of the corresponding letter, and $\rho_j$ denotes the longest factor consisting of letters greater than $i$. The other case is analogous.
\begin{gather}\label{EQ:form-case2.1}
\pi=\cdots \underbrace{\rho_1
	\mkern-17mu \overset{\overset{(k_i-1)\text{-th}}{\downarrow}}{i}\mkern-17mu
	\cdots}_{\alpha}
\rho_2
\mkern-5mu
\overset{\overset{k_i\text{-th}}{\downarrow}}{i}
\mkern-5mu
\cdots
\rho_3
\overset{\overset{k_{i-1}\text{-th}}{\downarrow}}{(i-1)}
\cdots
\ \mapsto\ 
%
%
\Phi(\pi)=\cdots 
\rho_2
\mkern-17mu
\overset{\overset{(k_i-1)\text{-th}}{\downarrow}}{i}
\mkern-17mu
\cdots
\underbrace{\rho_1
	\overset{\overset{k_{i-1}\text{-th}}{\downarrow}}{(i-1)}
	\cdots}_{\alpha}
\rho_3
\mkern-3mu
\overset{\overset{(k_{i-1}+1)\text{-th}}{\downarrow}}{(i-1)}
\mkern-3mu
\cdots
\end{gather}

\noindent\textbf{Case 2.2.} Assume that $i-1$ lies in $D(\pi,i,k_i-1)$ but not in $C(\pi,i,k_i-1)$. Let $\pi_{j_1}$ and $\pi_{j_2}$ be the first and the $(k_i-1)$-th occurrences of $i$, respectively. Let $\pi_{j_3}$ and $\pi_{j_4}$ be the first and the last~occurrences of $i-1$, respectively. Then $j_1\le j_2<j_3\le j_4$ and thus $\phi(j_1)\le \phi(j_2)<\phi(j_3)\le \phi(j_4)$. Define
\[
\alpha_1=\pi_{\phi(j_1)+1}\pi_{\phi(j_1)+2}\cdots\pi_{\phi(j_2)}\qquad\text{and}\quad
\alpha_2=\pi_{\phi(j_3)+1}\pi_{\phi(j_3)+2}\cdots\pi_{\phi(j_4)},
\]
where each $\alpha_k$ is empty if its starting index exceeds its ending index. To obtain $\Phi(\pi)$, we first remove the factors $\alpha_1$ and $\alpha_2$ from $\pi$. Note that the resulting permutation contains exactly two occurrences of $i$ and one occurrence of $i-1$. Then interchange the letters $i$ and $i-1$ simultaneously, and finally insert $\alpha_1$ after $\pi_{\phi(j_3)}$ and $\alpha_2$ after $\pi_{\phi(j_1)}$ (i.e., swap the positions of the two factors). The resulting permutation is $\Phi(\pi)$, which takes the following form, with the ordinal numbers and $\rho_j$'s denoted as in~\eqref{EQ:form-case2.1}.
\begin{gather}
\pi=\cdots \underbrace{\rho_1
	\overset{\overset{\text{1st}}{\downarrow}}{i}
	\cdots}_{\alpha_1}
\rho_2
\mkern-17mu \overset{\overset{(k_i-1)\text{-th}}{\downarrow}}{i}\mkern-17mu
\cdots
\underbrace{\rho_3
	\overset{\overset{\text{1st}}{\downarrow}}{(i-1)}
	\cdots}_{\alpha_2}
\rho_4
\overset{\overset{k_{i-1}\text{-th}}{\downarrow}}{(i-1)}
\cdots
\mkern-5mu
\overset{\overset{k_i\text{-th}}{\downarrow}}{i}
\mkern-5mu
\cdots\nonumber\\
\rotatebox[origin=c]{-90}{\ $\mapsto$\ }\nonumber\\
\Phi(\pi)=\cdots \underbrace{\rho_3
	\overset{\overset{\text{1st}}{\downarrow}}{(i-1)}
	\cdots}_{\alpha_2}
\rho_2
\overset{\overset{k_{i-1}\text{-th}}{\downarrow}}{(i-1)}
\cdots
\underbrace{\rho_1
	\overset{\overset{\text{1st}}{\downarrow}}{i}
	\cdots}_{\alpha_1}
\rho_4
\mkern-17mu
\overset{\overset{(k_{i}-1)\text{-th}}{\downarrow}}{i}
\mkern-17mu
\cdots
\mkern-3mu
\overset{\overset{(k_{i-1}+1)\text{-th}}{\downarrow}}{(i-1)}
\mkern-3mu
\cdots\label{EQ:form-case2.2}
\end{gather}

\noindent\textbf{Case 2.3.} Assume that $C(\pi,i,k_i-1)$ contains $i-1$. Let $\pi_{j_1}$ and $\pi_{j_3}$ be the first and the last occurrences of $i$ respectively, and let $\pi_{j_2}$ be the first occurrence of $i-1$. Then we have $j_1<j_2<j_3$, and hence $\phi(j_1)<\phi(j_2)<\phi(j_3)$. Set $\alpha_1=\pi_{\phi(j_1)+1}\pi_{\phi(j_1)+2}\cdots\pi_{\phi(j_2)}$ and $\alpha_2=\pi_{\phi(j_2)+1}\pi_{\phi(j_2)+2}\cdots\pi_{\phi(j_3)}$. Then $\Phi(\pi)$ is obtained from $\pi$ by interchanging the positions of the factors $\alpha_1$ and $\alpha_2$ and changing the letter of $\pi_{j_3}$ from $i$ to $i-1$. 
The explicit forms of $\pi$ and $\Phi(\pi)$ are shown below, where the ordinal numbers and $\rho_j$'s are as in~\eqref{EQ:form-case2.1}.
\begin{gather}
\pi=\cdots \underbrace{\rho_1
	\overset{\overset{\text{1st}}{\downarrow}}{i}
	\cdots
\rho_2
\mkern-17mu \overset{\overset{(k_i-1)\text{-th}}{\downarrow}}{i}\mkern-17mu
\cdots}_{\alpha_1}
\underbrace{\rho_3
	\overset{\overset{\text{1st}}{\downarrow}}{(i-1)}
	\cdots
\rho_4
\overset{\overset{k_{i-1}\text{-th}}{\downarrow}}{(i-1)}
\cdots}_{\alpha_2}
\rho_5
\mkern-5mu
\overset{\overset{k_i\text{-th}}{\downarrow}}{i}
\mkern-5mu
\cdots\nonumber\\
\rotatebox[origin=c]{-90}{\ $\mapsto$\ }\nonumber\\
\Phi(\pi)=\cdots \underbrace{\rho_3
	\overset{\overset{\text{1st}}{\downarrow}}{(i-1)}
	\cdots
\rho_4
\overset{\overset{k_{i-1}\text{-th}}{\downarrow}}{(i-1)}
\cdots
}_{\alpha_2}
\underbrace{\rho_1
	\overset{\overset{\text{1st}}{\downarrow}}{i}
	\cdots
\rho_2
\mkern-17mu
\overset{\overset{(k_{i}-1)\text{-th}}{\downarrow}}{i}
\mkern-17mu
\cdots
}_{\alpha_1}
\rho_5
\mkern-3mu
\overset{\overset{(k_{i-1}+1)\text{-th}}{\downarrow}}{(i-1)}
\mkern-3mu
\cdots\label{EQ:form-case2.3}
\end{gather}

The preceding definition ensures that $\Phi(\pi)$ is well-defined and belongs to $S_{M'}$. We now present several examples to illustrate the map $\Phi$ in concrete cases.

\begin{exa}\label{EX:Phi}
	Following Example~\ref{EX:childseq}, we apply $\Phi$ to $\pi$ iteratively 21 times and eventually obtain a permutation in $S_{\{1^9,2,3,4,5,6,7,8\}}$ as follows:
	\begin{align*}
	&5285146411733555\stackrel{{\cred\Phi}}{\longrightarrow}5285144641173355\stackrel{\Phi}{\longrightarrow}5285144464117335\stackrel{{\cred\Phi}}{\longrightarrow}4448415265117334\\
	\stackrel{\Phi}{\longrightarrow}\,&4448416511733424\stackrel{\Phi}{\longrightarrow}4448416511733234\stackrel{{\cred\Phi}}{\longrightarrow}7332344484165113\stackrel{\Phi}{\longrightarrow}7332344841651133\\
	\stackrel{\Phi}{\longrightarrow}\,&7332348416511333\stackrel{\Phi}{\longrightarrow}7332384165113333\stackrel{\Phi}{\longrightarrow}7332284165113333\stackrel{\Phi}{\longrightarrow}7332284165112333\\
	\stackrel{\Phi}{\longrightarrow}\,&7332284165112233\stackrel{\Phi}{\longrightarrow}7332284165112223\stackrel{\Phi}{\longrightarrow}7332284165112222\stackrel{{\cred \Phi}}{\longrightarrow}7232284165112222\\
	\stackrel{\Phi}{\longrightarrow}\,&7232284165111222\stackrel{{\cred\Phi}}{\longrightarrow}7232284165111122\stackrel{\Phi}{\longrightarrow}7232284165111112\stackrel{\Phi}{\longrightarrow}7232284165111111\\
	\stackrel{\Phi}{\longrightarrow}\,&7132284165111111\stackrel{\Phi}{\longrightarrow}7131284165111111
	\end{align*}
	
	We select five of these iterations (highlighted in the above display) to illustrate the map in detail.
	
	\begin{enumerate}
		\item We first compute $\Phi(5285146411733555)$, which falls under {\bf Case~2.1}. In this case, we have $j_1=15$, $j_2=16$ and $j_3=8$, which give $\phi(j_1)=14$, $\phi(j_2)=15$ and $\phi(j_3)=6$. Hence $\alpha = 5$, and we obtain $\Phi(5285146411733555)=528514{\cred 4}641173355$.
		
		\item The computation of $\Phi(5285144464117335)$ corresponds to {\bf Case~2.2}. Here $j_1=1$, $j_2=4$, $j_3=6$ and $j_4=10$, yielding $\phi(j_1)=0$, $\phi(j_2)=2$, $\phi(j_3)=5$ and $\phi(j_4)=8$. Thus $\alpha_1=52$ and $\alpha_2=444$. Removing these two factors gives $\pi' = {}_{\uparrow}851_{\uparrow}64117335$, which contains exactly two occurrences of $5$ and one occurrence of $4$. Interchanging $4$ and $5$ in $\pi'$ yields $_{\uparrow}841_{\uparrow}65117334$. Inserting $\alpha_1$ and $\alpha_2$ into each other's original positions, we obtain $\Phi(5285144464117335)={\cblue 444}841{\cred 52}65117334$.
		
		\item We now compute $\Phi(4448416511733234)$. With $i=4$, we are in {\bf Case~2.3}. We get $j_1=1$, $j_2=12$ and $j_3=16$, so $\phi(j_1)=0$, $\phi(j_2)=10$ and $\phi(j_3)=15$. Therefore $\alpha_1=4448416511$ and $\alpha_2=73323$, and we finally obtain $\Phi(4448416511733234)={\cblue 73323}{\cred 4448416511}3$.
		
		\item Taking $i=3$, we have $\Phi(7332284165112222)=7{\cred 2}32284165112222$ by {\bf Case~1.2}.
		
		\item The computation of $\Phi(7232284165111222)$ falls under $i=2$ and {\bf Case~1.1}. Direct application of the rule leads to $\Phi(7232284165111222)=7232284165111{\cred 1}22$.
	\end{enumerate}
\end{exa}

\begin{rem}
	If $i = n$, then $\theta$ acts as the identity map on $\oQ_M$. Moreover, if we further restrict to the subset of $\oQ_M$ consisting of permutations satisfying the condition of {\bf Case 2}, then our map $\Phi$ coincides with the Yan-Yang-Huang-Zhu bijection from~\cite{YanYangHuangZhu2022}. Though the two mappings differ in general, the underlying strategy for constructing $\Phi$ is based on the Yan-Yang-Huang-Zhu bijection. After handling the special situation of {\bf Case 1} separately, we first remove all letters greater than $i$ by applying the map $\theta$, then apply the Yan-Yang-Huang-Zhu bijection to obtain a preliminary permutation, and finally reinsert those larger letters into suitable positions so that the distribution of decreasing runs remains invariant.
\end{rem}

The remainder of this section is devoted to showing that $\Phi$ is indeed a bijection from $\oQ_M$ to $\oQ_{M'}$ which preserves the distribution of decreasing runs.

\begin{lem}\label{LEM:factor-keep}
	For any $\pi\in\oQ_M$, $\theta(\pi)$ contains the factor $i(i-1)$ if and only if $\theta(\Phi(\pi))$ does.
\end{lem}
\begin{proof}
	We proceed the proof by cases.
	
	In \textbf{Case 1}, the sequence $S(\pi,(i,i-1))$ (equivalently, $S(\theta(\pi),(i,i-1))$) is necessarily of one of the following two forms:
	\begin{align*}
	i\cdots i\,(i-1)\cdots(i-1)\,i\cdots i
	\quad\text{or}\quad
	\underbrace{(i-1)\cdots(i-1)}_{\text{may be empty}}\,i\cdots i\,(i-1)\cdots(i-1),
	\end{align*}
	corresponding to {\bf Case 1.1} and {\bf Case 1.2} respectively. Since $k_i>1$, the map's replacement of $i$'s by $i-1$'s leaves the unique factor $i(i-1)$ in $\theta(\pi)$ intact. Thus $\theta(\Phi(\pi))$ retains it as well.
	
	In \textbf{Case 2}, it suffices to show that $\theta(\Phi(\pi))$ has no factor $i(i-1)$. We first consider \textbf{Case 2.2} and \textbf{Case 2.3}. For \textbf{Case 2.2}, it is enough to note that there must be some letter smaller than $i-1$ between the last $i-1$ and the last $i$ in $\pi$. This follows immediately from the assumption that $i-1$ does not occur in $C(\pi,i,k_i-1)$. In \textbf{Case 2.3}, it is clear from~\eqref{EQ:form-case2.3} that $\theta(\pi)$ contains the factor $i(i-1)$ if and only if no letter smaller than $i-1$ lies between the $(k_i-1)$-th $i$ and the first $i-1$ in $\pi$; this condition is equivalent to the presence of the factor $i(i-1)$ in $\theta(\Phi(\pi))$.
	
	Now it remains to prove the lemma for \textbf{Case 2.1}. In this case, no $i-1$ occurs in the subsequence $\pi_{j_1}\pi_{j_1+1}\cdots\pi_{j_2}$, hence none occurs in $\alpha$. Thus a factor $i(i-1)$ in $\theta(\Phi(\pi))$ can only arise from:
	\begin{enumerate}
		\item the adjacent pair formed by $\pi_{\phi(j_3)}$ and the $i-1$ inside $\alpha$ after the replacement, which would require $\pi_{\phi(j_3)}=i$;
		\item the new adjacency created by removing $\alpha$, which would require $\pi_{j_4}=i-1$, where $j_4=\min\{j>\phi(j_2): \pi_j\le i\}$.
	\end{enumerate}
	The first possibility corresponds to the factor $\pi_{\phi(j_3)}\pi_{j_3}=i(i-1)$ in $\theta(\pi)$. The second implies $j_4=j_2$ and thus $\pi_{j_4}=i$. Therefore, both are impossible. This completes the proof.
\end{proof}

\begin{lem}
	The map $\Phi$ is a bijection from $\oQ_M$ to $\oQ_{M'}$.
\end{lem}
\begin{proof}	
	We first show that $\Phi(\pi)\in \oQ_{M'}$ for any $\pi\in \oQ_M$. Fix $\pi$. Since $k_{i+1}=k_{i+2}=\cdots=k_n=1$, it suffices to prove that $\theta(\Phi(\pi))$ is a quasi-Stirling permutation. We first consider \textbf{Case 1}. By Lemma~\ref{LEM:factor-keep}, both $\theta(\pi)$ and $\theta(\Phi(\pi))$ contain the factor $i(i-1)$. This implies that the block sets $B(\theta(\pi),(i,i-1))$ and $B(\theta(\Phi(\pi)),(i,i-1))$ coincide. Corollary~\ref{COR:test-2} tells us that every sequence in $B(\theta(\pi),(i,i-1))$ is a quasi-Stirling permutation and that no two of these sequences share a common letter. Hence the same holds for every sequence in $B(\theta(\Phi(\pi)),(i,i-1))$. Moreover, the replacement rule ensures that the sequence $S(\theta(\Phi(\pi)),(i,i-1))$ is a quasi-Stirling permutation. Applying Corollary~\ref{COR:test-2} once more, we conclude that $\theta(\Phi(\pi))$ is a quasi-Stirling permutation, and consequently so is $\Phi(\pi)$. The argument for \textbf{Case 2} is analogous and is therefore omitted.
	
	
	Now we show that $\Phi$ is a bijection between $\oQ_M$ and $\oQ_{M'}$. Let $\bar{\pi}\in\oQ_{M'}$. Once we can determine from which case $\bar{\pi}$ is mapped, its preimage can be easily recovered according to the construction of $\Phi$, which completes the proof. By Lemma~\ref{LEM:factor-keep}, one can first distinguish whether $\bar{\pi}$ comes from \textbf{Case 1} or \textbf{Case 2} by testing whether $\theta(\bar{\pi})$ contains the factor $i(i-1)$. If $\bar{\pi}$ arises from \textbf{Case 1}, then its reverse is straightforward. For \textbf{Case 2}, we observe the following characterizations.
	
	\noindent\textbf{Case 2.1} yields $\Phi(\pi)$ such that $D(\Phi(\pi),i-1,k_{i-1})$ does not contain $i$, as described in~\eqref{EQ:form-case2.1}.
	
	\noindent\textbf{Case 2.2} leads to $\Phi(\pi)$ such that $i$ lies in $D(\Phi(\pi),i-1,k_{i-1})$ but not in $C(\Phi(\pi),i-1,k_{i-1})$, as described in~\eqref{EQ:form-case2.2}.
	
	\noindent\textbf{Case 2.3} presents $\Phi(\pi)$ such that $i$ lies in $C(\Phi(\pi),i-1,k_{i-1})$, as described in~\eqref{EQ:form-case2.3}.
\end{proof}

%
%
%

\begin{thm}\label{THM:bijections-run}
	Let $M=\{1^{k_1},2^{k_2},\ldots,n^{k_n}\}$ be a multiset with $n\geq 2$ and $k_1,k_2,\ldots,k_n\in\mathbb P$. 
	Suppose that there exists an integer $i\in[n]\setminus\{1\}$ such that 
	$k_i>k_{i+1}=k_{i+2}=\cdots=k_n=1$. 
	Let 
	$
	M'=\{1^{k_1},\ldots,(i-1)^{k_{i-1}+1},i^{k_i-1},\ldots,n^{k_n}\}.
	$
	Then there exists a bijection $\Phi$ between $\overline{Q}_M$ and $\overline{Q}_{M'}$ such that
	$
	\DR(\pi)=\DR(\Phi(\pi))
	$
	for any $\pi\in \overline{Q}_M$.
\end{thm}
\begin{proof}
	By definition, $\Phi(\pi)$ is obtained through at most two operations: 
	first, rearranging the positions of certain factors, and then replacing some occurrences of the letters $i$ and $i-1$ with each other. Therefore, it is sufficient to show that these operations preserve the distribution of decreasing runs.
	
	For the first operation, note that the map $\phi$ guarantees that, in both the original and the resulting permutations, the entries immediately preceding and following each moved factor form ascents or levels with the factor's endpoints. Hence, the first operation neither destroys nor creates any decreasing runs.
	
	The second operation also preserves the distribution of decreasing runs. 
	Indeed, by Lemma~\ref{LEM:factor-keep}, no replacement is ever performed inside a factor of the form $i(i-1)$. 
	Moreover, when forming a decreasing run, the letters $i$ and $i-1$ are indistinguishable from all other letters.
\end{proof}

Applying Equations~\eqref{EQ:p-DR} to the above theorem yields the following corollary immediately.

\begin{cor}\label{COR:bijections-pattern}
	Let $M=\{1^{k_1},2^{k_2},\ldots,n^{k_n}\}$ be a multiset with $n\geq2$ and $k_1,k_2,\ldots,k_n\in\bP$. Suppose that there exists an integer $i\in[n]\setminus\{1\}$ such that $k_{i}>k_{i+1}=k_{i+2}=\cdots=k_n=1$. Let $M'=\{1^{k_1},\ldots,(i-1)^{k_{i-1}+1},i^{k_{i}-1},\ldots,n^{k_n}\}$. Then there exists a bijection $\Phi$ between $\overline{Q}_M$ and $\overline{Q}_{M'}$ such that for any $\pi\in\oQ_M$, 
	\[(\underline{1},\underline{21},\ldots,\underline{n\cdots21})\pi=(\underline{1},\underline{21},\ldots,\underline{n\cdots21})\Phi(\pi).\]
\end{cor}

As shown in Example~\ref{EX:Phi}, applying $\Phi$ for $\sum_{j=2}^i(j-1)(k_j-1)$ times, we finally obtain a bijection between $\oQ_M$ and $\oQ_{\{1^{K-n+1},2,\ldots,n\}}$ with $K=k_1+k_2+\cdots+k_n$ and thus prove Theorem~\ref{THM:bijections-run-pattern}.

\begin{exa}
	Following Example~\ref{EX:Phi}, one can verify that throughout the iterative application of $\Phi$, every resulting permutation $\pi'$ preserves the same distribution of decreasing runs as $\pi$. Specifically, $\DR(\pi)=\DR(\pi')=(6,2,2,0,0,0,0,0)$. The same holds for the number of occurrences of decreasing consecutive patterns, namely, $$(\underline{1},\underline{21},\ldots,\underline{8\cdots21})\pi=(\underline{1},\underline{21},\ldots,\underline{8\cdots21})\pi'=(16,6,2,0,0,0,0,0).$$
\end{exa}

\begin{rem}
	Even if we take $M'$ to be the multiset obtained from $M$ by merely permuting the multiplicities of the letters, the joint distribution of the patterns $\underline{321}$ and $\underline{11}$ over $\oQ_M$ and $\oQ_{M'}$ can differ. For example, let $M=\{1^2,2,3\}$ and $M'=\{1,2^2,3\}$. The permutation $\pi=3211\in\oQ_M$ satisfies $(\underline{321},\underline{11})\pi=(1,1)$, whereas there is no permutation $\pi'\in\oQ_{M'}$ such that $(\underline{321},\underline{11})\pi'=(1,1)$. The same phenomenon occurs for the patterns $\underline{321}$ and $\underline{12}$, since $\{\pi\in\oQ_M \mid (\underline{321},\underline{12})\pi=(1,1)\}=\{1321\}$, while $\{\pi\in\oQ_{M'} \mid (\underline{321},\underline{12})\pi=(1,1)\}=\{2321,3212\}$. In this sense, both the Yan-Yang-Huang-Zhu bijection in~\cite{YanYangHuangZhu2022} and ours are the best possible, as they preserve the joint distribution for as many patterns as possible.
\end{rem}

\section{Distribution Functions for Decreasing Statistics}\label{SEC:dis}

This section is devoted to the distribution functions for decreasing runs and decreasing consecutive patterns over quasi-Stirling permutations. Using the bijection from Section~\ref{SEC:bijections}, it suffices to study these distributions over permutations of a multiset in which only the letter $1$ may have multiplicity greater than one. These can be obtained via the well-known run theorem.

\subsection{Distribution for Decreasing Runs}\label{SEC:dis-DR}

In this section, we compute the distribution functions for decreasing runs over quasi-Stirling permutations. We present the associated recurrence relations and the corresponding exponential generating function, thereby proving Theorem~\ref{THM:dis-run}.

We begin by recalling a related elegant result for general multipermutations, known as the \textit{run theorem}. This result has been established by Gessel~\cite[Theorem~5.2]{Gessel1977}, by Jackson and Aleliunas~\cite[Theorem~4.1]{Jackson1977}, by Goulden and Jackson~\cite[Theorem~4.2.3]{Goulden1983}, and restated by Gessel and Zhuang~\cite[Theorem~11]{GesselZhuang2014}. With the shorthand notations $\vt:=(t_j)_{j\ge1}$ and $\vx:=(x_j)_{j\ge1}$, we define the generating function for the distribution of decreasing runs over all words with positive integer letters by
\[
\cR_M(\vt;\vx)
:= \sum_{n \ge 1} \sum_{k_1,k_2,\ldots,k_n \ge 0}
\left( \sum_{\substack{\pi \in S_{M}\\ M=\{1^{k_1}, 2^{k_2}, \ldots, n^{k_n}\}}} \mathbf{t}^{\DR(\pi)} \right)
x_1^{k_1} x_2^{k_2} \cdots x_n^{k_n},
\]
where $\mathbf{t}^{\DR(\pi)} = t_1^{m_1} t_2^{m_2} \cdots t_n^{m_n}$ whenever $ \DR(\pi) = (m_1,m_2,\ldots,m_n) $. The run theorem then provides a complete characterization of this generating function.

\begin{lem}[Run Theorem]
	The generating function $ \cR_M(\vt;\vx) $ admits the explicit form
	\begin{align}\label{EQ:GDR}
	\cR_M(\vt;\vx)
	= \left( 1 -\sum_{n=1}^{\infty} W_n(\vt) e_n(\vx) \right)^{-1},
	\end{align}
	where $ e_n(\vx) $ denotes the elementary symmetric function of degree $ n $ in the variables $ x_1, x_2, \ldots $, and $ W_n(\vt) $ is the function of $ t_1, t_2, \ldots, t_n $ determined by the expansion~\eqref{EQ:W_n}.
\end{lem}

In this paper, we are concerned with the distribution of decreasing runs over multipermutations of the multiset $\{1^k,2,\ldots,n\}$. Recall that we denote the corresponding distribution function by $\ccr(\vt;n,k)$ and it is indeed the coefficient of $x_1^k x_2 \cdots x_n$ in $\cR_M(\vt;\vx)$. Using symmetric function techniques from~\cite{Yang2020}, we derive the following recurrence relation.

\begin{prop}\label{PROP:rec-dr}
	Let $W_n(\vt)$ denote the function of $t_1,t_2,\ldots,t_n$ determined by the expansion~\eqref{EQ:W_n}. Then for any $n\geq1$ and $k\geq1$, we have
	\begin{align}\label{EQ:rec-dr}
	\ccr(\vt;n,k)=\sum_{m=1}^{n}W_m(\vt)\left(\binom{n-1}{m}\ccr(\vt;n-m,k)+\binom{n-1}{m-1}\ccr(\vt;n-m+1,k-1)\right).
	\end{align}
	Moreover, $\ccr(\vt;n,0)=\ccr(\vt;n-1,1)$ for $n\geq2$, with the initial conditions $\ccr(\vt;1,0)=1$ and $\ccr(\vt;n,k)=0$ whenever $n\leq 0$ or $k< 0$.
\end{prop}
\begin{proof}
	Denote by $\cc(\vt;\vk)$ the coefficient of the monomial $x_1^{k_1}x_2^{k_2}\cdots x_n^{k_n}$ in $\cR_M(\vt;\vx)$, where $\vk$ is the vector $(k_1,k_2,\ldots,k_n)$. Multiplying both sides of Equation~\eqref{EQ:GDR} by $ 1 -\sum_{n=1}^{\infty} W_n(\vt) e_n(\vx) $, and then comparing the coefficients of the monomial $x_1^{k_1}x_2^{k_2}\cdots x_n^{k_n}$ on each side, we obtain 
	\begin{align}\label{EQ:gen-rec}
	\cc(\vt;\vk)-\left(\sum_{m=1}^{n}W_m(\vt)\sum_{\vv\in V_{n,m}}\cc(\vt;\vk-\vv)\right)=0
	\end{align}
	for $n\geq1$, where $V_{n,m}$ is the set of $0$-$1$ vectors of length $n$ with exactly $m$ 1's. Furthermore, the function $\cR_M(\vt;\vx)$ is symmetric in $\vx$ by virtue of~\eqref{EQ:GDR}, which implies that $\cc(\vt;\vk)=\cc(\vt;\vk')$ for any $\vk'$ obtained from $\vk$ by permuting its entries. Therefore, Equation~\eqref{EQ:gen-rec} can be further simplified. Specializing to the case $\vk=(k,1,\ldots,1)$ yields the desired recurrence relation~\eqref{EQ:rec-dr}.
\end{proof}

The recurrence relation~\eqref{EQ:rec-dr} allows us to derive the exponential generating function $\cR(\vt;x,y)$ of the sequence $\ccr(\vt;n,k)$, see~\eqref{EQ:def-cR} for the definition of $\cR(\vt;x,y)$. 
To this end, we first recall that Gessel and Zhuang~\cite{GesselZhuang2014} have shown that applying the homomorphism $e_n(\vx)\mapsto x^n/(n!)$ to~\eqref{EQ:GDR} yields the following equation for the exponential generating function $\cW(\vt;x)$ of $W_n(\vt)$.


\begin{lem}[\cite{GesselZhuang2014}, Lemma 9]\label{LEM:EG-constant}
	The exponential generating function for the distribution of decreasing runs over classical permutations is given by
	\[
	1+\sum_{n=1}^{\infty}\frac{\ccr(\vt;n,1)}{n!}x^n=\frac{1}{1-\cW(\vt;x)}.
	\]
\end{lem}

Now we are ready to compute $\cR(\vt;x,y)$ by Proposition~\ref{PROP:rec-dr} and Lemma~\ref{LEM:EG-constant}.

\begin{prop}\label{PROP:cR}
	The exponential generating function $\cR(\vt;x,y)$ has the explicit form
	\begin{align}\label{EQ:cR}
	\cR(\vt;x,y)=\frac{x}{1-\cW(\vt;x)-y\cdot \cW'(\vt;x)},
	\end{align}
	where the prime denotes differentiation with respect to $x$.
\end{prop}
\begin{proof}
	The constant term of $\cR$ in $y$ is
	\begin{align}\label{EQ:EG-constant}
	\cR(\vt;x,0)&=\sum_{n=1}^{\infty}\frac{\ccr(\vt;n,0)}{(n-1)!}x^n
	=x+\sum_{n=2}^{\infty}\frac{\ccr(\vt;n-1,1)}{(n-1)!}x^n\nonumber\\
	&=x+\sum_{n=1}^{\infty}\frac{\ccr(\vt;n,1)}{n!}x^{n+1}
	=\frac{x}{1-\cW(\vt;x)},
	\end{align}
	where the last equality follows from Lemma~\ref{LEM:EG-constant}. A direct computation then gives
	\begin{align}\label{EQ:EG-other}
	&\left(\cW(\vt;x)-1\right)\left(\cR(\vt;x,y)-\cR(\vt;x,0)\right)+y\cdot \cW'(\vt;x)\,\cR(\vt;x,y)\\
	=\,&\left(\sum_{n=1}^{\infty}\frac{W_n(\vt)}{n!}x^n-1\right)\sum_{n\geq1,k\geq1}\frac{\ccr(\vt;n,k)}{(n-1)!}x^ny^k+\sum_{n=1}^{\infty}\frac{W_n(\vt)}{(n-1)!}x^{n-1}y\sum_{n\geq1,k\geq0}\frac{\ccr(\vt;n,k)}{(n-1)!}x^ny^k\nonumber\\
	=\,&\sum_{n\geq1,k\geq1}\left(\sum_{m=1}^{n}W_m(\vt)\left(\frac{\ccr(\vt;n-m,k)}{m!(n-m-1)!}+\frac{\ccr(\vt;n-m+1,k-1)}{(m-1)!(n-m)!}\right)-\frac{\ccr(\vt;n,k)}{(n-1)!}\right)x^ny^k=0,\nonumber
	\end{align}
	where the last equality uses Equation~\eqref{EQ:rec-dr}. Substituting~\eqref{EQ:EG-constant} to~\eqref{EQ:EG-other} yields the desired result.
\end{proof}

\begin{exa}
	We now compute the functions $\ccr(\vt;n,k)$ for several special values of $n$ and $k$ using Propositions~\ref{PROP:rec-dr} and~\ref{PROP:cR}.
	\begin{enumerate}
		\item For $n=1$, the recurrence relation~\eqref{EQ:rec-dr} gives
		\[
		\ccr(\vt;1,k)=W_1(\vt)\,\ccr(\vt;1,k-1)=W_1^k(\vt)\ccr(\vt;1,0)=W_1^k(\vt).
		\]
		Since $W_1(\vt)=t_1$ from~\eqref{EQ:W_n}, we obtain $\ccr(\vt;1,k)=t_1^k$, which corresponds to the unique multipermutation of the multiset $\{1^k\}$.
		
		\item For $k=1$, the recurrence relation~\eqref{EQ:rec-dr} reduces to
		\begin{align*}
		\ccr(\vt;n,1)
		&=W_n(\vt)+\sum_{m=1}^{n-1}W_m(\vt)\left(\binom{n-1}{m}\ccr(\vt;n-m,1)+\binom{n-1}{m-1}\ccr(\vt;n-m,1)\right)\\
		&=W_n(\vt)+\sum_{m=1}^{n-1}W_m(\vt)\binom{n}{m}\ccr(\vt;n-m,1).
		\end{align*}
		Note that this can also be obtained from Lemma~\ref{LEM:EG-constant}. Iterating it yields
		\begin{align*}
		\begin{array}{llll}
		&\ccr(\vt;1,1)=t_1,\qquad\qquad\qquad
		&\ccr(\vt;2,1)=t_1^2 + t_2,\\
		&\ccr(\vt;3,1)=t_1^3 + 4 t_1 t_2 + t_3,
		&\ccr(\vt;4,1)=t_1^4 + 11 t_1^2 t_2 + 6 t_1 t_3 + 5 t_2^2 + t_4,\ \cdots\\
		\end{array}
		\end{align*}
		which give the distributions of decreasing runs in classical permutations. 
		
		\item Finally, we compute $\ccr(\vt;n,n+1)$ for $n=1,2,3$ by expanding the left-hand side of~\eqref{EQ:cR}. From~\eqref{EQ:W_n}, we have $W_1(\vt)=t_1$, $W_2(\vt)=-t_1^2+t_2$ and $W_3(\vt)=t_1^3 - 2t_1t_2 + t_3$. Then we get
		\[\cW(\vt;x)=t_1x+\frac{-t_1^2+t_2}{2}x^2+\frac{t_1^3 - 2t_1t_2 + t_3}{6}x^3+\cdots\]
		and thus
		\begin{align*}
		&\cR(\vt;x,y)=\frac{x}{1-\cW(\vt;x)-y\cdot \cW'(\vt;x)}\\
		=\,&\frac{x}{1-\left(t_1x+\frac{-t_1^2+t_2}{2}x^2+\cdots\right)-y \left(t_1+\left(-t_1^2+t_2\right)x+\frac{t_1^3 - 2t_1t_2 + t_3}{2}x^2+\cdots\right)}\\
		=\,&x\sum_{m=0}^{6}\left(t_1y+\left(t_1-t_1^2y+t_2y\right)x+\left(-\frac{1}{2}t_1^2+\frac{1}{2}t_2+\frac{1}{2}t_1^3y-t_1t_2y+\frac{1}{2}t_3y\right)x^2\right)^m+\cdots\\
		=\,&\cdots+t_1^2xy^2+\cdots+\left(t_1^4 + 3t_1^2t_2\right)x^2y^3+\cdots+\frac{t_1^6 + 13t_1^4t_2 + 4t_1^3t_3 + 12t_1^2t_2^2}{2!}x^3y^4+\cdots.
		\end{align*}
		
		As shown in Theorem~\ref{THM:dis-run}, $\ccr(\vt;n,n+1)$ indeed gives the distribution of decreasing runs in quasi-Stirling permutations of the multiset $\{1^2,2^2,\ldots,n^2\}$.
	\end{enumerate}
\end{exa}

Combining Theorem~\ref{THM:bijections-run-pattern} with Propositions~\ref{PROP:rec-dr} and~\ref{PROP:cR} completes the proof of Theorem~\ref{THM:dis-run}.

\subsection{Distribution for Decreasing Consecutive Patterns}\label{SEC:dis-p}

This section presents the recurrence relations and generating functions for the distribution functions of decreasing consecutive patterns over quasi-Stirling permutations, including joint patterns, single patterns and pattern avoidance. While these results are direct consequences of those for decreasing runs in Section~\ref{SEC:dis-DR}, we record them here, which reproves some known results and yields Corollary~\ref{COR:most-p} as a further consequence.

We first clarify the precise definitions of these distribution functions. For a multiset $M=\{1^{k_1},2^{k_2},\ldots,n^{k_n}\}$, let $\mathcal{p}(\vt;\oQ_M)$ denote the joint distribution function of decreasing consecutive patterns over $\oQ_M$, defined by
\[\cp(\vt;\oQ_M):=\sum_{\pi\in \oQ_M} t_1^{(\underline{1})\pi}t_2^{(\underline{21})\pi}\cdots t_n^{(\underline{n\cdots 21})\pi}. \]
In particular, for every positive integer $s$, let $\cp_s(t;\oQ_M)$ denote the distribution function of the pattern $\underline{s\cdots 21}$ over $\oQ_M$. In other words, $\cp_s(t;\oQ_M)$ is obtained from $\cp(\vt;\oQ_M)$ by setting $t_s=t$ and $t_j=1$ for all $j\neq s$. By Theorem~\ref{THM:bijections-run-pattern}, it suffices to consider the corresponding functions on $S_{M'}$, where $M'$ is the multiset $\{1^{k},2,\ldots,n\}$. To this end, define 
\[\cp(\vt;n,k):=\sum_{\pi\in S_{M'}} t_1^{(\underline{1})\pi}t_2^{(\underline{21})\pi}\cdots t_n^{(\underline{n\cdots 21})\pi}\quad\text{and}\quad 
\cp_s(t;n,k):=\cp(1,\ldots,\mkern-3mu\underset{\underset{s\text{-th}}{\uparrow}}{t}\mkern-4mu,1,\ldots;n,k). \]
Let $\cP(\vt;x,y)$ and $\cP_s(t,x,y)$ denote the exponential generating functions of $\cp(\vt;n,k)$ and $\cp_s(t;n,k)$ respectively, given by
\[\cP(\vt;x,y):=\sum_{n\geq1,k\geq0}\frac{\cp(\vt;n,k)}{(n-1)!}x^n y^k
\quad\text{and}\quad
\cP_s(t,x,y):=\sum_{n\geq1,k\geq0}\frac{\cp_s(t;n,k)}{(n-1)!}x^n y^k.\]
Then the recurrence relation and exponential generating function for $\cp(\vt;n,k)$ follow directly from Theorem~\ref{THM:dis-run} by replacing each $t_j$ with $t_1^j t_2^{j-1}\cdots t_j$ for all positive integers $j$.

\begin{thm}\label{THM:enum-joint-pattern}
	Let $M=\{1^{k_1},2^{k_2},\ldots,n^{k_n}\}$ be a multiset with $n\geq 1$ and $k_1,k_2,\ldots,k_n\in\mathbb P$. Let 
	$K=k_1+k_2+\cdots +k_n$. Then we have
	\[
	\cp(\vt;\oQ_M)=\cp(\vt;n,K-n+1).
	\]
	Furthermore, let $ \bar{W}_n(\vt) $ be the function of $ t_1, t_2, \ldots, t_n $ determined by the expansion
	\begin{align}\label{EQ:barW_n}
	1 -\sum_{n=1}^{\infty} \bar{W}_n(\vt) z^n
	= \left( 1 + \sum_{n=1}^{\infty} t_1^nt_2^{n-1}\cdots t_n z^n \right)^{-1}.
	\end{align}
	Then $\cp(\vt;n,k)$ satisfies the recurrence relations
	\begin{align*}
	\cp(\vt;n,k)=
	\sum_{m=1}^{n}\bar{W}_m(\vt)\left(\binom{n-1}{m}\cp(\vt;n-m,k)+\binom{n-1}{m-1}\cp(\vt;n-m+1,k-1)\right)
	\end{align*}
	for $n\geq1$ and $k\geq1$, together with $\cp(\vt;n,0)=\cp(\vt;n-1,1)$ for $n\geq2$,
	and initial conditions $\cp(\vt;1,0)=1$ and $\cp(\vt;n,k)=0$ whenever $n\leq0$ or $k<0$.
	Let $\bar{\cW}(\vt;x)$ be the exponential generating function of $\bar{W}_n(\vt)$ defined as
	$\bar{\cW}(\vt;x):=\sum_{n=1}^{\infty}\bar{W}_n(\vt)x^n/n!$.
	Then the exponential generating function of $\cp(\vt;n,k)$ is given by
	\begin{align*}
	\cP(\vt;x,y)=\frac{x}{1-\bar{\cW}(\vt;x)-y\cdot \bar{\cW}'(\vt;x)},
	\end{align*}
	where the prime denotes differentiation with respect to $x$.
\end{thm}

The equations for $\cp_s(t;\oQ_M)$ have a more explicit form, since the substitution makes the ordinary generating function of $\tilde{W}_{s,n}(t)$ a rational function. Let $\tilde{W}_{s,n}(t)$ and $\tilde{\cW}_s(t,x)$ be the functions obtained from $\bar{W}_n(\mathbf{t})$ and $\bar{\mathcal{W}}(\mathbf{t};x)$, respectively, by setting $t_s=t$ and $t_j=1$ for all $j\neq s$. Then by definition~\eqref{EQ:barW_n}, we have
\begin{align}
\sum_{n=1}^{\infty} \tilde{W}_{s,n}(t) z^n
&=1- \left( \sum_{n=0}^{s-1}z^n + \sum_{n=s}^{\infty} t^{n-s+1} z^n \right)^{-1}
=1- \left( \frac{1-z^s}{1-z} + \frac{1}{t^{s-1}}\left(\frac{1}{1-tz}-\sum_{n=0}^{s-1}(tz)^n\right) \right)^{-1}\nonumber\\
&=1-\frac{(1-tz)(1-z)}{(1-tz)(1-z^s)+\frac{1-z}{t^{s-1}}\left(1-\sum_{n=0}^{s-1}(tz)^n+\sum_{n=1}^{s}(tz)^n\right)}
\nonumber\\
&=\frac{(t-1)z^s-tz^2+z}{1-tz+(t-1)z^s}.\label{EQ:barW-gf}
\end{align}
Therefore, the exponential generating function $\tilde{\cW}_{s}(t,x)$ can be expressed as 
\begin{align*}
\tilde{\cW}_{s}(t,x)=\frac{(t-1)x^s-tx^2+x}{1-tx+(t-1)x^s}\odot_x e^x,
\end{align*}
where the notation $\odot_x$ denotes the Hadamard product with respect to $x$. Recall that the \textit{Hadamard product} of two power series $f(x)=\sum_{n=0}^{\infty}f_nx^n$ and $g(x)=\sum_{n=0}^{\infty}g_nx^n$ with respect to $x$ is defined as $f(x)\odot_x g(x):=\sum_{n=0}^{\infty}f_ng_nx^n$.

The ordinary generating function~\eqref{EQ:barW-gf} also yields the recurrence relation and an explicit expression for the functions $\tilde{W}_{s,n}(t)$. If $s=1$, then $\tilde{W}_{1,1}(t)=t$ and $\tilde{W}_{1,n}(t)=0$ for all $n\geq2$. Now assume that $s>1$. Then we have
\begin{align}\label{EQ:barW-rr}
\tilde{W}_{s,n}(t)=t\tilde{W}_{s,n-1}(t)-(t-1)\tilde{W}_{s,n-s}(t)
\end{align}
for $n\geq s+1$, with initial conditions $\tilde{W}_{s,1}(t)=1$, $\tilde{W}_{s,n}(t)=0$ for $2\leq n \leq s-1$, and $\tilde{W}_{s,s}(t)=t-1$. Moreover, expanding the right-hand side of~\eqref{EQ:barW-gf} as a power series in $z$ yields
\begin{align*}
&\,\tilde{W}_{s,n}(t)=\sum_{m=1}^{\lfloor \frac{n}{s}\rfloor}(-1)^{m-1}\binom{n-sm+m-1}{m-1}t^{n-sm}(t-1)^m\\
&\,-\sum_{m=0}^{\lfloor \frac{n-2}{s}\rfloor}(-1)^{m}\binom{n-sm+m-2}{m}t^{n-sm-1}(t-1)^m
+\sum_{m=0}^{\lfloor \frac{n-1}{s}\rfloor}(-1)^{m}\binom{n-sm+m-1}{m}t^{n-sm-1}(t-1)^m
\end{align*}
for $n\geq2$. Combining the last two sums via the identity $\binom{n}{m}-\binom{n-1}{m}=\binom{n-1}{m-1}$, we obtain
\begin{align*}
\tilde{W}_{s,n}(t)=&\,\sum_{m=1}^{\lfloor \frac{n}{s}\rfloor}(-1)^{m-1}\binom{n-sm+m-1}{m-1}t^{n-sm}(t-1)^m\\
&\,+\sum_{m=1}^{\lfloor \frac{n-2}{s}\rfloor}(-1)^{m}\binom{n-sm+m-2}{m-1}t^{n-sm-1}(t-1)^m
+ \left(\lfloor \frac{n-1}{s}\rfloor-\lfloor \frac{n-2}{s}\rfloor-1\right)(1-t)^{\lfloor \frac{n-1}{s}\rfloor}.
\end{align*}

In terms of the functions $\tilde{W}_{s,n}(t)$ and $\tilde{\cW}_s(t,x)$, Theorem~\ref{THM:enum-joint-pattern} yields the following characterization of the functions $\cp_s(t;n,k)$.

\begin{thm}\label{THM:enum-p-dis}
	Let $M=\{1^{k_1},2^{k_2},\ldots,n^{k_n}\}$ be a multiset with $n\geq 1$ and $k_1,k_2,\ldots,k_n\in\mathbb P$. Let 
	$K=k_1+k_2+\cdots +k_n$ and $s$ be a positive integer. Then we have
	\[
	\cp_s(t;\oQ_M)=\cp_s(t;n,K-n+1).
	\]
	Furthermore, $\cp_s(t;n,k)$ satisfies the recurrence relations
	\begin{align}\label{EQ:cps-rr}
	\cp_s(t;n,k)=
	\sum_{m=1}^{n}\tilde{W}_{s,m}(t)\left(\binom{n-1}{m}\cp_s(t;n-m,k)+\binom{n-1}{m-1}\cp_s(t;n-m+1,k-1)\right) 
	\end{align}
	for $n\geq s$ and $k\geq1$, and $\cp_s(t;n,0)=\cp_s(t;n-1,1)$ for $n\geq2$.
	The initial conditions are $\cp_s(t;n,k)=(n+k-1)!/k!$ for $1\leq n\leq s-1$ and $k\geq0$, and $\cp_s(t;n,k)=0$ whenever $n\leq0$ or $k<0$.
	The exponential generating function of $\cp_s(t;n,k)$ is given by
	\begin{align*}
	\cP_s(t,x,y)=\frac{x}{1-\tilde{\cW}_s(t,x)-y\cdot \tilde{\cW}_s'(t,x)},
	\end{align*}
	where the prime denotes differentiation with respect to $x$.
\end{thm}

This theorem implies several classical enumerative results. For example, setting $k=1$, we obtain the distribution functions for classical permutations with the recurrence relation
\begin{align}
\cp_s(t;n,1)=
n\cp_s(t;n-1,1)+\sum_{m=s}^{n-1}\tilde{W}_{s,m}(t)\binom{n}{m}\cp_s(t;n-m,1)+\tilde{W}_{s,n}(t)\label{EQ:pstn1}
\end{align}
for $n\geq s$. In particular, when $n=s$, we have
\begin{align}\label{EQ:psts1}
\cp_s(t;s,1)=s\cp_s(t;s-1,1)+\tilde{W}_{s,s}(t)=t+s!-1,
\end{align}
which agrees with the fact that among classical permutations of length $s$, exactly one contains the pattern $\underline{s\cdots 21}$. In addition, Theorem~\ref{THM:enum-p-dis} also prepares us for the proof of Corollary~\ref{COR:most-p}.

\begin{proof}[Proof of Corollary~\ref{COR:most-p}]
	It suffices to show that, for any $k\geq1$ and $n\geq s$, the polynomial $\cp_s(t;n,k)$ has degree $n-s+1$ with leading coefficient $k^{n-1}$ for $s=2$, and $k$ for $s>2$. The proof then follows by induction on $n\geq1$.
	
	For $n=s$, the recurrence relation~\eqref{EQ:cps-rr} reduces to
	\[\cp_s(t;s,k)=(s-1)\cp_s(t;s-1,k)+\cp_s(t;s,k-1)+(t-1).\]
	Therefore, the recursive computation with the initial condition~\eqref{EQ:psts1} yields that the leading term of $\cp_s(t;s,k)$ is exactly $kt$.
	
	For $n>s$, assume that the result holds for smaller $n$. Fixing $n$, we then proceed by induction on $k\geq1$. Note that for $n\geq s$, the leading term of $\tilde{W}_{s,n}(t)$ is $t^{n-s+1}$ by Equation~\eqref{EQ:barW-rr}. For $k=1$, it is straightforward to see that the leading term of $\cp_s(t;n,1)$ is $t^{n-s+1}$, either from the combinatorial interpretation or from the recurrence relation~\eqref{EQ:pstn1}. Then for $k>1$, assume that the result holds for smaller $k$. If $s>2$, the recurrence relation~\eqref{EQ:cps-rr} implies that the leading term of $\cp_s(t;n,k)$ is contributed by 
	\[\binom{n-1}{0}\tilde{W}_{s,1}(t)\cp_s(t;n,k-1)+\binom{n-1}{n-1}\tilde{W}_{s,n}(t)\cp_s(t;1,k-1),\]
	and is therefore $(k-1+1)t^{n-s+1}$. If $s=2$, then the leading term of $\cp_2(t;n,k)$ is contributed by
	\[\sum_{m=1}^{n}\tilde{W}_{2,m}(t)\binom{n-1}{m-1}\cp_2(t;n-m+1,k-1),\]
	and is thus
	$\sum_{m=1}^{n}t^{m-2+1}\binom{n-1}{m-1}(k-1)^{n-m}t^{n-m+1-2+1}
	=k^{n-1}t^{n-1}$.
\end{proof}

We conclude this section by presenting the enumerative results on pattern avoidance, a topic that has attracted considerable attention and been extensively investigated. In~\cite[Example 3, p.~51]{Gessel1977}, Gessel also examined this problem as an meaningful application of the run theorem. Setting $t=0$, we observe that the only indices $(s,n)$ for which $\tilde{W}_{s,n}(0)$ is nonzero are those satisfying either $s \mid n$, in which case $\tilde{W}_{s,n}(0) = -1$, or $s \mid (n-1)$, in which case $\tilde{W}_{s,n}(0) = 1$. Theorem~\ref{THM:enum-p-dis} then specializes to the following form.

\begin{thm}\label{THM:avoid-p}
	Let $M=\{1^{k_1},2^{k_2},\ldots,n^{k_n}\}$ be a multiset with $n\geq 1$ and $k_1,k_2,\ldots,k_n\in\mathbb P$. Let 
	$K=k_1+k_2+\cdots +k_n$ and $s\geq2$ be a positive integer. Then the number of the quasi-Stirling permutations in $\oQ_M$ avoiding pattern $\underline{s\cdots 21}$ is $\cp_s(0;n,K-n+1)$. Furthermore, the sequence $\cp_s(0;n,k)$ satisfies the recurrence relations
	\begin{align*}
	\cp_s(0;n,k)=&\,\sum_{m=1}^{\lfloor 
		\frac{n}{s}\rfloor}\left(-\left(\binom{n-1}{sm}\cp_s(0;n-sm,k)+\binom{n-1}{sm-1}\cp_s(0;n-sm+1,k-1)\right)\right.\\
	&\,\left.+\left(\binom{n-1}{sm+1}\cp_s(0;n-sm-1,k)+\binom{n-1}{sm}\cp_s(0;n-sm,k-1)\right)\right)\\
	&\,+\left((n-1)\cp_s(0;n-1,k)+\cp_s(0;n,k-1)\right)
	\end{align*}
	for $n\geq s$ and $k\geq1$, and $\cp_s(0;n,0)=\cp_s(0;n-1,1)$ for $n\geq2$.
	The initial conditions are $\cp_s(0;n,k)=(n+k-1)!/k!$ for $1\leq n\leq s-1$ and $k\geq0$, and $\cp_s(0;n,k)=0$ whenever $n\leq0$ or $k<0$.	
	Let $\hat{\cW}_{s}(x)$ be defined as $\hat{\cW}_{s}(x)=(-x^s+x)/(1-x^s)\odot_x e^x$.
	Then the exponential generating function of $\cp_s(0;n,k)$ is given by
	\begin{align*}
	\cP_s(0,x,y)=\frac{x}{1-\hat{\cW}_{s}(x)-y\cdot \hat{\cW}_{s}'(x)},
	\end{align*}
	where the prime denotes differentiation with respect to $x$.
\end{thm}

\begin{rem}
	Yang and Zeilberger~\cite{Yang2020} also presented several elegant expressions for the distribution of increasing consecutive patterns in multipermutations over general multisets. They focused on the patterns of length at least two. Although they employed the Goulden--Jackson cluster method, which differs from that used in Gessel's work~\cite{Gessel1977}, their resulting expressions are structurally identical to those obtained from the run theorem, with the weight function $\tilde{W}^*_{s,n}(t)$ satisfying the recurrence
	\begin{align*}
	\tilde{W}^*_{s,n}(t):=
	(t-1)\sum_{m=1}^{s-1}\tilde{W}^*_{s,n-m}(t) 
	\end{align*}
	for $n\geq s+1$, together with the same initial values as those of $\tilde{W}_{s,n}(t)$. An inductive argument shows that $\tilde{W}^*_{s,n}(t)=\tilde{W}_{s,n}(t)$ for all $n\ge 1$, since
	\begin{align*}
	\tilde{W}_{s,n}(t)&=t\tilde{W}_{s,n-1}(t)-(t-1)\tilde{W}_{s,n-s}(t)
	=t\tilde{W}^*_{s,n-1}(t)-(t-1)\tilde{W}^*_{s,n-s}(t)\\
	&=(t-1)\tilde{W}^*_{s,n-1}(t)+(t-1)\sum_{m=1}^{s-1}\tilde{W}^*_{s,n-m-1}(t)-(t-1)\tilde{W}^*_{s,n-s}(t)\\
	&=(t-1)\sum_{m=1}^{s-1}\tilde{W}^*_{s,n-m}(t)=\tilde{W}^*_{s,n}(t).
	\end{align*}
	Hence, Theorems~\ref{THM:enum-p-dis} and~\ref{THM:avoid-p} can also be obtained from Theorems 3 and 1 in~\cite{Yang2020} respectively by an argument analogous to those in Propositions~\ref{PROP:rec-dr} and~\ref{PROP:cR}.
\end{rem}

Taking $k=1$, we obtain the number of classical permutations avoiding the decreasing consecutive pattern of length $s$, which satisfies the recurrence
\begin{align*}
\cp_s(0;n,1)
&= n\cp_s(0;n-1,1)
-\sum_{m=1}^{\lceil n/s \rceil-1} \binom{n}{sm}\cp_s(0;n-sm,1)
-\left(\left\lfloor \frac{n}{s} \right\rfloor - \left\lceil \frac{n}{s} \right\rceil \right)\times1\\
&\quad +\sum_{m=1}^{\lceil (n-1)/s \rceil-1} \binom{n}{sm+1}\cp_s(0;n-sm-1,1)
+\left(\left\lfloor \frac{n-1}{s} \right\rfloor - \left\lceil \frac{n-1}{s} \right\rceil \right)\times1
\end{align*}
for $n\ge s$. This recurrence originally goes back to David and Barton~\cite{David1962}. The corresponding sequences for $3\le s\le 9$ can be found in the OEIS~\cite{oeis} under the identifiers A001212, A117158, A177523, A177533, A177553, A230051 and A230231, respectively.




\section*{Acknowledgments} The author is grateful to Shaoshi Chen and Zhicong Lin for their valuable suggestions.

\section*{Declaration of AI Assistance}

During the preparation of this manuscript, the author utilized an AI-assisted tool for language polishing and grammar checking to improve the readability and clarity of the text. All core reasoning and key conclusions were independently completed by the author. The use of AI did not involve the substantive generation or creative contribution to the research content. The author takes full academic responsibility for the entire manuscript.

\bibliographystyle{plain}

\bibliography{qStirling}

@book {Gessel1977,
    AUTHOR = {Gessel, I. M.},
     TITLE = {Generating functions and enumeration of sequences},
PUBLISHER = {PhD thesis, Massachusetts Institute of Technology}, 
      YEAR = {1977},
     PAGES = {(no paging)},
   MRCLASS = {Thesis},
  MRNUMBER = {2940769},
       URL =
              {http://gateway.proquest.com/openurl?url_ver=Z39.88-2004&rft_val_fmt=info:ofi/fmt:kev:mtx:dissertation&res_dat=xri:pqdiss&rft_dat=xri:pqdiss:0322960},
}

@article {Lin2021,
    AUTHOR = {Lin, Z. and Ma, J. and Zhang, P. B.},
     TITLE = {Statistics on multipermutations and partial
              {$\gamma$}-positivity},
   JOURNAL = {J. Combin. Theory Ser. A},
  FJOURNAL = {Journal of Combinatorial Theory. Series A},
    VOLUME = {183},
      YEAR = {2021},
     PAGES = {Paper No. 105488, 24},
      ISSN = {0097-3165,1096-0899},
   MRCLASS = {05A05 (05A15 05A19)},
  MRNUMBER = {4268724},
MRREVIEWER = {Eric\ S.\ Egge},
       DOI = {10.1016/j.jcta.2021.105488},
       URL = {https://doi.org/10.1016/j.jcta.2021.105488},
}

@article {Bona2008,
    AUTHOR = {B\'{o}na, M.},
     TITLE = {Real zeros and normal distribution for statistics on
              {S}tirling permutations defined by {G}essel and {S}tanley},
   JOURNAL = {SIAM J. Discrete Math.},
  FJOURNAL = {SIAM Journal on Discrete Mathematics},
    VOLUME = {23},
      YEAR = {2008/09},
    NUMBER = {1},
     PAGES = {401--406},
      ISSN = {0895-4801,1095-7146},
   MRCLASS = {05A05 (05A15 05A16)},
  MRNUMBER = {2476838},
MRREVIEWER = {Daniel\ E.\ Warren},
       DOI = {10.1137/070702254},
       URL = {https://doi.org/10.1137/070702254},
}

@incollection {Janson2008,
    AUTHOR = {Janson, S.},
     TITLE = {Plane recursive trees, {S}tirling permutations and an urn
              model},
 BOOKTITLE = {Fifth {C}olloquium on {M}athematics and {C}omputer {S}cience},
    SERIES = {Discrete Math. Theor. Comput. Sci. Proc., AI},
     PAGES = {541--547},
 PUBLISHER = {Assoc. Discrete Math. Theor. Comput. Sci., Nancy},
      YEAR = {2008},
   MRCLASS = {60C05 (05A05 05C05 68R10)},
  MRNUMBER = {2508813},
MRREVIEWER = {Hsien-Kuei\ Hwang},
}

@article {Gessel1978,
    AUTHOR = {Gessel, I. and Stanley, R. P.},
     TITLE = {Stirling polynomials},
   JOURNAL = {J. Combinatorial Theory Ser. A},
  FJOURNAL = {Journal of Combinatorial Theory. Series A},
    VOLUME = {24},
      YEAR = {1978},
    NUMBER = {1},
     PAGES = {24--33},
      ISSN = {0097-3165},
   MRCLASS = {05A15},
  MRNUMBER = {462961},
MRREVIEWER = {Stephen\ Tanny},
       DOI = {10.1016/0097-3165(78)90042-0},
       URL = {https://doi.org/10.1016/0097-3165(78)90042-0},
}

@article{oeis,
  title={The on-line encyclopedia of integer sequences},
  author={Sloane, Neil J. A. and others},
  journal={Published electronically at \url{https://oeis. org}},
  year={2018}
}

@book {David1962,
    AUTHOR = {David, F. N. and Barton, D. E.},
     TITLE = {Combinatorial chance},
 PUBLISHER = {Hafner Publishing Co., New York},
      YEAR = {1962},
     PAGES = {ix+356},
   MRCLASS = {62.00 (62.10)},
  MRNUMBER = {155371},
MRREVIEWER = {Z.\ W.\ Birnbaum},
}

@article {Archer2019-1,
    AUTHOR = {Archer, K. and Gregory, A. and Pennington, B. and
              Slayden, S.},
     TITLE = {Pattern restricted quasi-{S}tirling permutations},
   JOURNAL = {Australas. J. Combin.},
  FJOURNAL = {The Australasian Journal of Combinatorics},
    VOLUME = {74},
      YEAR = {2019},
     PAGES = {389--407},
      ISSN = {1034-4942,2202-3518},
   MRCLASS = {05A05 (05A15)},
  MRNUMBER = {3969738},
MRREVIEWER = {Yan\ Zhuang},
}

@article {Brenti1989-2,
    AUTHOR = {Brenti, F.},
     TITLE = {Unimodal, log-concave and {P}\'{o}lya frequency sequences in
              combinatorics},
   JOURNAL = {Mem. Amer. Math. Soc.},
  FJOURNAL = {Memoirs of the American Mathematical Society},
    VOLUME = {81},
      YEAR = {1989},
    NUMBER = {413},
     PAGES = {viii+106},
      ISSN = {0065-9266,1947-6221},
   MRCLASS = {05A15 (05A10 05A20)},
  MRNUMBER = {963833},
MRREVIEWER = {Ira\ Gessel},
       DOI = {10.1090/memo/0413},
       URL = {https://doi.org/10.1090/memo/0413},
}

@article {Dzhumadil2014-5,
    AUTHOR = {Dzhumadil'daev, A. and Yeliussizov, D.},
     TITLE = {Stirling permutations on multisets},
   JOURNAL = {European J. Combin.},
  FJOURNAL = {European Journal of Combinatorics},
    VOLUME = {36},
      YEAR = {2014},
     PAGES = {377--392},
      ISSN = {0195-6698,1095-9971},
   MRCLASS = {05A05 (05A15)},
  MRNUMBER = {3131903},
MRREVIEWER = {Vincent\ Vatter},
       DOI = {10.1016/j.ejc.2013.08.002},
       URL = {https://doi.org/10.1016/j.ejc.2013.08.002},
}

@article {Brenti1998-3,
    AUTHOR = {Brenti, F.},
     TITLE = {Hilbert polynomials in combinatorics},
   JOURNAL = {J. Algebraic Combin.},
  FJOURNAL = {Journal of Algebraic Combinatorics. An International Journal},
    VOLUME = {7},
      YEAR = {1998},
    NUMBER = {2},
     PAGES = {127--156},
      ISSN = {0925-9899,1572-9192},
   MRCLASS = {13D40 (05A15 16W99)},
  MRNUMBER = {1609885},
MRREVIEWER = {Ralf\ Fr\"{o}berg},
       DOI = {10.1023/A:1008656320759},
       URL = {https://doi.org/10.1023/A:1008656320759},
}

@book {StanleyBook2,
    AUTHOR = {Stanley, R. P.},
     TITLE = {Enumerative combinatorics. {V}ol 2},
    SERIES = {Cambridge Studies in Advanced Mathematics},
 PUBLISHER = {Cambridge University Press, Cambridge},
      YEAR = {1999},
      VOLUME = {62},
     PAGES = {xiv+626},
      ISBN = {978-1-107-60262-5},
   MRCLASS = {05-02 (05A15 06-02)},
  MRNUMBER = {2868112},
}

@article {Yang2020,
    AUTHOR = {Yang, M. and Zeilberger, D.},
     TITLE = {Increasing consecutive patterns in words},
   JOURNAL = {J. Algebraic Combin.},
  FJOURNAL = {Journal of Algebraic Combinatorics. An International Journal},
    VOLUME = {51},
      YEAR = {2020},
    NUMBER = {1},
     PAGES = {89--101},
      ISSN = {0925-9899,1572-9192},
   MRCLASS = {05A15 (05A05)},
  MRNUMBER = {4064613},
MRREVIEWER = {Donatella\ Merlini},
       DOI = {10.1007/s10801-018-0868-5},
       URL = {https://doi.org/10.1007/s10801-018-0868-5},
}

@book {Goulden1983,
    AUTHOR = {Goulden, I. P. and Jackson, D. M.},
     TITLE = {Combinatorial enumeration},
    SERIES = {Wiley-Interscience Series in Discrete Mathematics},
 PUBLISHER = {John Wiley \& Sons, Inc., New York},
      YEAR = {1983},
     PAGES = {xxiv+569},
      ISBN = {0-471-86654-7},
   MRCLASS = {05-02 (05A15)},
  MRNUMBER = {702512},
MRREVIEWER = {R.\ W.\ Robinson},

}

@article {GesselZhuang2014,
    AUTHOR = {Gessel, I. M. and Zhuang, Y.},
     TITLE = {Counting permutations by alternating descents},
   JOURNAL = {Electron. J. Combin.},
  FJOURNAL = {Electronic Journal of Combinatorics},
    VOLUME = {21},
      YEAR = {2014},
    NUMBER = {4},
     PAGES = {Paper 4.23, 21},
      ISSN = {1077-8926},
   MRCLASS = {05A15 (05A05)},
  MRNUMBER = {3292260},
MRREVIEWER = {Shi-Mei\ Ma},
       DOI = {10.37236/4624},
       URL = {https://doi.org/10.37236/4624},
}

@article {Jackson1977,
    AUTHOR = {Jackson, D. M. and Aleliunas, R.},
     TITLE = {Decomposition based generating functions for sequences},
   JOURNAL = {Canadian J. Math.},
  FJOURNAL = {Canadian Journal of Mathematics. Journal Canadien de
              Math\'{e}matiques},
    VOLUME = {29},
      YEAR = {1977},
    NUMBER = {5},
     PAGES = {971--1009},
      ISSN = {0008-414X,1496-4279},
   MRCLASS = {05A15},
  MRNUMBER = {450080},
MRREVIEWER = {Michael\ Henle},
       DOI = {10.4153/CJM-1977-098-3},
       URL = {https://doi.org/10.4153/CJM-1977-098-3},
}

@article {YanZhu2022,
    AUTHOR = {Yan, S. H. F. and Zhu, X.},
     TITLE = {Quasi-{S}tirling polynomials on multisets},
   JOURNAL = {Adv. in Appl. Math.},
  FJOURNAL = {Advances in Applied Mathematics},
    VOLUME = {141},
      YEAR = {2022},
     PAGES = {Paper No. 102415, 14},
      ISSN = {0196-8858,1090-2074},
   MRCLASS = {05A05 (05C30)},
  MRNUMBER = {4470615},
MRREVIEWER = {Mikl\'{o}s\ B\'{o}na},
       DOI = {10.1016/j.aam.2022.102415},
       URL = {https://doi.org/10.1016/j.aam.2022.102415},
}

@article {YanYangHuangZhu2022,
    AUTHOR = {Yan, S. H. F. and Yang, L. and Huang, Y. and Zhu,
              X.},
     TITLE = {Statistics on quasi-{S}tirling permutations of multisets},
   JOURNAL = {J. Algebraic Combin.},
  FJOURNAL = {Journal of Algebraic Combinatorics. An International Journal},
    VOLUME = {55},
      YEAR = {2022},
    NUMBER = {4},
     PAGES = {1265--1277},
      ISSN = {0925-9899,1572-9192},
   MRCLASS = {05A15 (05A05)},
  MRNUMBER = {4423530},
       DOI = {10.1007/s10801-021-01093-z},
       URL = {https://doi.org/10.1007/s10801-021-01093-z},
}

@article {Elizalde2021,
    AUTHOR = {Elizalde, S.},
     TITLE = {Descents on quasi-{S}tirling permutations},
   JOURNAL = {J. Combin. Theory Ser. A},
  FJOURNAL = {Journal of Combinatorial Theory. Series A},
    VOLUME = {180},
      YEAR = {2021},
     PAGES = {Paper No. 105429, 35},
      ISSN = {0097-3165,1096-0899},
   MRCLASS = {05A05 (05A15 05A19)},
  MRNUMBER = {4211010},
MRREVIEWER = {Istv\'{a}n\ Mez\H{o}},
       DOI = {10.1016/j.jcta.2021.105429},
       URL = {https://doi.org/10.1016/j.jcta.2021.105429},
}

@book{Kitaev2011,
    AUTHOR = {Kitaev, S.},
     TITLE = {Patterns in permutations and words},
    SERIES = {Monographs in Theoretical Computer Science. An EATCS Series},
 PUBLISHER = {Springer, Heidelberg},
      YEAR = {2011},
     PAGES = {xxii+494},
      ISBN = {978-3-642-17332-5; 978-3-642-17333-2},
   MRCLASS = {05-02 (05A05 05A15 68-02 68R05 68R15 68W32)},
  MRNUMBER = {3012380},
MRREVIEWER = {Sergi\ Elizalde},
       DOI = {10.1007/978-3-642-17333-2},
       URL = {https://doi.org/10.1007/978-3-642-17333-2},
}

\end{document}